\documentclass[11pt,a4paper]{amsart}
\usepackage{amsthm,amsmath,amscd}
\title{Tower sets and other configurations with the Cohen-Macaulay property }
\author{Giuseppe Favacchio, Alfio Ragusa
	\and  Giuseppe Zappal\`a}
\subjclass[2010]{13 H 10, 14 N 20, 13 D 40}

\DeclareSymbolFont{rsfscript}{OMS}{rsfs}{m}{n}
\DeclareSymbolFontAlphabet{\mathrsfs}{rsfscript}

\DeclareSymbolFont{AMSb}{U}{msb}{m}{n}
\DeclareSymbolFontAlphabet{\mathbb}{AMSb}

\DeclareSymbolFont{eufrak}{U}{euf}{m}{n}
\DeclareSymbolFontAlphabet{\gothic}{eufrak}

\newcommand{\f}{\footnotesize}

\newcommand{\id}{\operatorname{id}}

\newcommand{\pp}{\mathbb P}

\newcommand\depth{\operatorname{depth}}

\newcommand\rw{\Rightarrow}

\newcommand\prm{{\gothic p}}

\newcommand\zz{{\mathbb Z}}

\newcommand{\ff}{\mathcal F}
\newcommand{\mb}{\mathcal M}
\newcommand{\sd}{\mathcal S}
\newcommand{\zp}{\zz_{+}}
\newcommand{\zc}{\zz_+^c}
\newcommand{\zd}{(\zz_+^2)^*}

\newcommand{\pip}{\pi_1(T)\cap\pi_2(T)}
\newcommand\pd{\operatorname{proj-dim}}
\newcommand{\xu}{\underline{x}}

\newcommand{\hu}{\underline{h}}

\newtheorem{thm}{Theorem}[section]
\newtheorem{lem}[thm]{Lemma}
\newtheorem{prp}[thm]{Proposition}
\newtheorem{cor}[thm]{Corollary}

\theoremstyle{definition}

\newtheorem{dfn}[thm]{Definition}

\theoremstyle{remark}
\newtheorem{rem}[thm]{Remark}
\newtheorem{exm}[thm]{Example}

\newcounter{num}

\begin{document}

%\maketitle
%footnotemark
%\footnotetext[1]{2000 {\it Mathematics subject classification: }13 D 40, 13 H 10.}
%\subjclass{13 D 40, 13 H 10}

\begin{abstract}
Some well-known arithmetically Cohen-Macaulay configurations of linear varieties in $\pp^r$ as $k$-configurations, partial intersections and star configurations are generalized by introducing {\sl tower schemes}. Tower schemes are reduced schemes that are finite union of linear varieties whose support set is a suitable finite subset of $\zc$ called {\sl tower set}.
We prove that the tower schemes are arithmetically Cohen-Macaulay and we compute their Hilbert function in terms of their support.
%\par
Afterwards, since even in codimension $2$ not every arithmetically Cohen-Macaulay squarefree monomial ideal is the ideal of a tower scheme, we slightly extend this notion by defining {\sl generalized tower schemes}  (in codimension $2$) and we show that the support of these configurations (the {\sl generalized tower set}) gives a combinatorial characterization of the primary decomposition of the arithmetically Cohen-Macaulay squarefree monomial ideals.
\end{abstract}

%\maketitle
%footnotemark
%\footnotetext[1]{2000 {\it Mathematics subject classification: }13 D 40, 13 H 10.}
%\subjclass{13 D 40, 13 H 10}

%\begin{abstract}
%\end{abstract}

\maketitle

\section*{Introduction}
\markboth{\it Introduction}{\it Introduction}
In the last few years a large number  of researchers in algebraic
geometry  in order to produce projective schemes with suitable Hilbert functions
and graded Betti numbers constructed special configurations of linear varieties related to some subsets of $\zc.$ Among these should be cited the  partial intersection schemes  introduced first in \cite{MR} and generalized in any codimension in \cite{RZ1} and the $k$-configurations defined  in \cite{GS} and \cite{GHS} to obtain maximal graded Betti numbers with respect to a fixed Hilbert function. On the other hand, to study the extremal Hilbert functions for fat point schemes in the plane, secant varieties of some classical algebraic varieties and some properties of the symbolic powers of ideals, the star configurations were defined and deeply investigated (see for instance \cite{AS}, \cite{GHM}). All these configurations lead to aCM ideals, mostly  monomial and squarefree. Looking at what all these configurations have in common, in this paper we define the tower sets (Definition \ref{twse}), suitable finite subsets of  $\zc,$ on which are supported the tower schemes (Definition \ref{twsc}), which generalize all the previous mentioned configurations. 
These tower sets enclose the combinatorial aspects of such configurations.
\par
Also for these schemes we are able to prove that they have the aCM property (Theorem \ref{tower-aCM}). 
Moreover, we compute the Hilbert function of the tower schemes in terms of its tower set support.
At this point one can believe that, at least for monomial squarefree ideals, all aCM ideals can be constructed in this way. Unfortunately, already in codimension $2,$ this is false as we show  in Example \ref{ex}. So the question which arises is to find the right configuration which could characterize all the aCM monomial squarefree ideals in a polynomial ring. Here we give a complete answer in codimension $2$ (Theorems \ref{gtsacm}, \ref{um} and \ref{gts-acm}) defining a slight modification of the tower schemes (generalized tower sets and schemes, see Definitions \ref{tg} and \ref{gts}). The codimension bigger than $2$ case remains open. 

After preliminaries and basic facts, in section 2  we introduce tower sets and tower schemes and we prove that all these schemes are aCM (Theorem \ref{tower-aCM}). Then we show that every tower scheme has the same Hilbert function as a corresponding tower scheme supported on a left segment whose Hilbert function was computed in \cite{RZ1} (see Proposition \ref{LD} and Corollary \ref{HFtow}). Section 3 is devoted to give a combinatorial characterization for aCM squarefree monomial ideals of codimension 2. To do that we give a slight generalization of tower sets and tower schemes (Definitions \ref{tg} and \ref{gts}). Then we prove numerous preparatory results about these sets and schemes and finally in Theorems \ref{gtsacm} and \ref{gts-acm} we prove the stated characterization.

\section{Notation and preliminaries} 
\markboth{\it Notation and preliminaries}
{\it Notation and preliminaries}
Throughout the paper $k$ will be a field and $R:=k[x_1,\ldots,x_n]=\oplus_d R_d$ will be the standard graded polynomial $k$-algebra.
\par
We will denote by $\zp:=\{r\in\zz\mid r>0\}.$ If $r\in\zp$ we will set $[r]:=\{1,\ldots,r\}.$
If $c,r\in\zp$ we will denote by $C_{c,r}$ the set of the subsets of $[r]$ of cardinality $c.$
\par
Moreover, we will set $\pi_i:\zp^c\to \zp$ the projection on the $i$-th component.
On the set $\zc$ we will use the following standard partial order. If $\alpha,\beta\in\zc,$ $\alpha\le\beta$ iff $\pi_i(\alpha)\le\pi_i(\beta)$ for every $i\in [c].$
\par
We will denote by $(\zc)^*:=\{(a_1,\ldots,a_c)\in\zc\mid a_i\ne a_j \text{ for every }i\ne j\}.$
Let $T\subset\zc$ be a finite set. Let $1\le t\le c-1$ be an integer and let $\alpha\in\zp^t.$ We set
 $$T_{\alpha}:=\{\gamma\in\zp^{c-t}\mid (\gamma,\alpha)\in T\}$$
and 
 $$T^{\alpha}:=\{\gamma\in\zp^{c-t}\mid (\alpha,\gamma)\in T\}$$
\begin{dfn}
The function $\varphi:(\zc)^*\to C_{c,n}$ such that $\varphi(a_1,\ldots,a_c)=\{a_1,\ldots,a_c\},$ will be called {\em forgetful} function.
A function $\omega:C_{c,n}\to (\zc)^*$ will be called {\em ordinante} iff $\varphi\circ\omega=\id_{C_{c,n}}.$
\end{dfn}
Let $L\subset\zc$ be a finite set. $L$ is said {\em left segment} if for every $\alpha\in L$ and $\beta\in\zc$ with $\beta\le\alpha$ it follows that $\beta\in L.$
\par
Let $L\subset\zc$ be a left segment. The set $\{\alpha_1,\ldots,\alpha_r\}\subseteq L$ is called {\em set of generators} for $L$ if for every $\alpha\in L,$ $\alpha\le\alpha_i$ for some $i.$
The element $(\max\pi_1(L),\ldots,\max\pi_c(L))\in\zp^c$ is said the {\em size} of $L.$ 
\par
Let $L\subset\zc$ be a left segment of size $(m_1,\ldots,m_c),$ with $c<n.$ For $1\le i\le c,$ let $\ff_i=\{f_{i1},\ldots,f_{im_i}\}$ be $c$ families of generic linear forms belonging to $R.$ For every $\alpha=(a_1,\ldots,a_c)\in L$ we set $I_{\alpha}:=(f_{1a_1},\ldots,f_{ca_c}).$ We recall that the scheme defined by the ideal $I_L(\ff_1,\ldots,\ff_c):=\bigcap_{\alpha\in L}I_{\alpha}$ is called {\em partial intersection}, with support on $L$ and with respect to the families 
$\ff_1,\ldots,\ff_c.$
\par
If $\alpha\in\zc$ we set $v(\alpha):=\sum_{i=1}^c\pi_i(\alpha).$
If $L\subset\zc$ is a left segment, the Hilbert function of $L$ is 
$$H_L(i):=|\{\alpha\in L\mid v(\alpha)=i+c\}|.$$ 
We remind that $H_L$ coincides with the Hilbert function of a partial intersection supported on $L$ (for details about left segments and partial intersections see \cite{RZ1}).
\par
In the sequel, if $M$ is a matrix of rank $r,$ with entries in $R,$ we will denote by $I(M)$ the ideal generated by the minors of size $r$ in $M.$

\section{Towers sets} 
\markboth{\it Towers sets}{\it Towers sets}
Many recent papers dealt with special configurations of linear subvarieties of projective spaces which raised up to Cohen-Macaulay varieties,
for instance partial intersections studied in \cite{RZ1}, $k$-configurations studied in \cite{GHS}, star configurations studied in \cite{GHM}.
In this section we would like to generalize all these configurations in such a way to preserve the Cohen-Macaulayness.
\par
\begin{dfn}\label{twse}
%Note that for $t=0,$ $(\zz^+)^{0}=\{\emptyset\}$ and $T_{\emptyset}=T.$
%Let $2\le t\le c+1$ be an integer and $(u_{t},\ldots,u_c)\in(\zz^+)^{c-t+1}.$ We set
% $$T_{u_{t}, \ldots, u_c}:=\{(y_{1}, \ldots, y_{t-1})\in(\zz^+)^{t-1}\mid (y_{1},\ldots,y_{t-1},u_{t},\ldots,u_c)\in T\}.$$
%Note that for $t=c+1,$ $(u_{c+1},\ldots,u_c)=\emptyset$ and $T_{\emptyset}=T.$
Let $T\subset\zc$ be a finite set. 
We say that $T$ is a {\em tower set} if for every $t\in [c-1]$ and for every $\alpha,\beta\in\zp^{t},$ with $\alpha<\beta,$ $T_{\alpha}\ne\emptyset,$ we have 
$T_{\alpha}\supseteq T_{\beta}.$
% $$T_{h,u_{t}, \ldots, u_c} \supseteq T_{k,u_{t}, \ldots, u_c}\,\,\text{ whenever }h\le k \text{ and }T_{h,u_{t}, \ldots, u_c}\ne\emptyset.$$
 %$$h\le k,\,h,k\in\pi_{t-1}(T).$$
\end{dfn}
Note that when $c=1$ every finite subset of $\zp$ is a tower set.
\begin{rem}
If $T\subset\zc$ is a tower set and $\alpha\in\zp^t$ then $T_{\alpha}\subset\zp^{c-t}$ is also a tower set.
\end{rem}

\begin{dfn}\label{twsc}
Let $T\subset\zp^c$ be a tower set. Let $R:=k[x_1,\ldots,x_n],$ with $2 \le c\le n-1.$ Let $\ff_i=\{f_{ij}\mid j\in\pi_i(T)\},$ $1\le i\le c,$ where $f_{ij}\in R_{d_{ij}},$ such that 
$f_{ij}$ and $f_{ih}$ are coprime when %$j,h\in\pi_i(T),$ 
$j\ne h$ and for every $\alpha=(a_1,\ldots,a_c)\in T$ the sequence $(f_{1a_1},\ldots,f_{ca_c})$ is regular. We will denote by $I_{\alpha}$ the complete intersection ideal generated by $f_{1a_1},\ldots,f_{ca_c}.$ We set
 $$I_T(\ff_1,\ldots,\ff_c):=\bigcap_{\alpha\in T}I_{\alpha}.$$
It defines a $c$-codimensional subscheme of $\pp^n$ called {\em tower scheme}, with support on $T,$ with respect to the families $\ff_1,\ldots,\ff_c.$
\end{dfn}

Note that if $T$ is a $c$-left segment then $T$ is a tower set, so every partial intersection is a tower scheme.
\par
Recently many people investigated special subschemes called star configurations. We recall that a star configuration is defined as follows. Let $R:=k[x_1,\ldots,x_n],$ $s,c\in\zp$ such that $c\le\min\{s,n-1\}.$ Take a set $\ff$ consisting of $s$ forms $f_1,\ldots,f_s\in R$ such that any $c$ of them are a regular sequence. If $s\ge a_1>\ldots>a_c\ge 1$ are integers and $\alpha=\{a_1,\ldots,a_c\}$ we set $I_{\alpha}:=(f_{a_1},\ldots,f_{a_c}).$ A star configuration is the subscheme $V_c(\ff,\pp^n)\subset\pp^n$ defined by the ideal $\bigcap I_{\alpha}$ where $\alpha$ runs over all the subsets of $[s]$ of cardinality $c.$ For more details on star configurations see, for instance, \cite{GHM}.
\begin{rem}
A star configuration is a particular tower scheme. Namely, let
 $$T=\{(a_1,\ldots,a_c)\in\zp^c\mid s\ge a_1>\ldots>a_c\ge 1\},$$
and let us consider the families of forms
 $$\ff_i=(f_{s-i+1},\ldots,f_{c-i+1}),\,\,\,\text{ for }1\le i\le c.$$
$T$ is trivially a tower set and $V(I_T(\ff_1,\ldots,\ff_c))=V_c(\ff,\pp^n).$
\end{rem}

In the quoted papers it was shown that partial intersections, $k$-configurations and star configurations are all aCM schemes.
Now we prove that every tower scheme is an aCM scheme and this generalizes those results. We need the following lemma, which is
a slight generalization of Lemma 1.6 in \cite{RZ1}.

\begin{lem}\label{l1}
Let $c,r\ge 2$ be integers. Let $V_1 \supseteq \ldots \supseteq V_r$ be
$(c-1)$-codimensional aCM subschemes of
${\pp}^n$  and $A_j=V(f_j)$ hypersurfaces, $1 \le j \le r,$  with $\deg f_j=d_j.$ 
We set $Y_i:=V_i \cap A_i$ and let us suppose that $Y_i$
 is $c$-codimensional for each $i$ and that $Y_i$ and $Y_j$
 have no common components for $i \ne j.$ We set $d :=\sum_{i=1}^r d_i,$ $Y:=Y_1\cup \ldots
 \cup Y_{r-1}$ and $X:=Y \cup Y_r.$ Then the following
sequence of graded $R$-modules
$$ 0 \to   I_{Y_r}(-(d-d_r))  \stackrel{f}{\rightarrow}  I_X  \stackrel{\varphi}{\rightarrow} I_Y/(f)  \to   0 $$
is exact, where $f=\prod \limits_{i=1}^{r-1}f_i$ 
and $\varphi$ is the natural map. Moreover
$$I_X=I_{V_1}+ f_1I_{V_2}+ f_1f_2I_{V_3}+ \cdots +
f_1\ldots f_{r-1}I_{V_r}+(f_1 \ldots f_r).$$
\end{lem}
\begin{proof}
The proof is analogous to that of Lemma 1.6 in \cite{RZ1}.
We report it for convenience of the reader.
\par
Observe that the exactness of the above sequence in the middle
depends on the fact that $f$ is regular modulo $I_{Y_r},$
since $Y_i$ and $Y_j$ have no common components for $i \ne j.$
So, the only not trivial fact to prove is  that the map $\varphi$ is
surjective. For this we use induction on $r.$ For $r=2,$ since $V_1$ is aCM,
$I_Y=I_{V_1} + (f_1),$ therefore every element
in $I_Y/(f_1)$ looks like $z+(f_1)$ with $z \in I_{V_1} \subseteq I_{V_2}.$
Hence, $z \in I_{Y_1} \cap I_{Y_2}=I_X.$ So, $\varphi$ is surjective and
the sequence is exact. Now, from the exactness of this sequence it
follows that $I_X$ is generated by  $I_{V_1}$  and  $f_1I_{Y_2}$,
i.e. $I_X=I_{V_1}+f_1I_{V_2}+(f_1f_2).$ \par
Let us suppose the lemma true for $r-1.$ This means, in particular,
that $I_Y=I_{V_1}+f_1I_{V_2}+ \cdots + f_1 \ldots f_{r-2}I_{V_{r-1}}+
(f_1 \ldots f_{r-1}).$ Therefore, every element $z\in I_{Y}/(f_1 \ldots
f_{r-1})$ has the form $x+(f_1 \ldots f_{r-1})$ with $x \in
I_{V_1}+f_1I_{V_2}+ \cdots + f_1 \ldots f_{r-2}I_{V_{r-1}}.$ Hence,
$x \in I_{V_r}$ which implies $x \in I_Y \cap I_{Y_r}=I_X.$
Again, by the exactness of our sequence we get that $I_X$ is
generated by $f_1\ldots f_{r-1}I_{Y_r}$ and by
$I_{V_1}+f_1I_{V_2}+ \cdots + f_1 \ldots f_{r-2}I_{V_{r-1}}$, i.e.
$I_X=I_{V_1}+f_1I_{V_2}+ \cdots + f_1 \ldots f_{r-1}I_{V_{r}}+
(f_1 \ldots f_r).$
\end{proof}

%\begin{cor}\label{hlem}
%With the terminology of Lemma \ref{l1}, 
%\end{cor}

We are ready to prove our result.

\begin{thm}\label{tower-aCM}
Every tower scheme is aCM.
\end{thm}
\begin{proof}
Let $X$ be a tower scheme of codimension $c.$
To show that $X$ is aCM we use induction on $c.$ The property is
trivially true for $c=1,$ so we can assume that every tower scheme 
of codimension $c-1$ is aCM. Let $T$ be the support of $X$ and let $\ff_i=\{f_{ij}\mid j\in\pi_i(T)\},$ for $1\le i \le c,$ be the families defining $X.$
Let $\pi_c(T)=\{a_1,\ldots,a_s\},$ with $a_1<\ldots<a_s.$
For every $i\in [s]$ we denote by $V_i$ the tower scheme of codimension $c-1$ supported
on $T_{a_i}.$ If $i<j$ then $V_i\supseteq V_j.$ By inductive hypotheses each $V_i$ is aCM. Moreover we denote by $A_i$ the hypersurface defined by $f_{ca_i},$ with $i\in [s]$ and $Y_i=V_i\cap A_i.$ Note that by the hypotheses on $\ff_h$'s $Y_i$ is aCM of codimension $c.$
%Let $\pi_c(T)=\{a_1,\ldots,a_s\},$ with $a_1<\ldots<a_s.$
\par
Therefore we have $X=\bigcup \limits_{1\le i\le s}Y_i.$ 
Now we use induction on $s.$ For $s=1$ $X=Y_1$ is aCM. Suppose that  $Y=
\bigcup \limits_{1 \le i \le {s-1}}Y_i$ is aCM and show that
$X=Y \cup Y_s$ is aCM. Applying the previous lemma we get
the exact sequence
$$0 \to  I_{Y_s}(-\deg f)  \to I_X \to I_Y/(f) \to   0 $$
where $f=\prod \limits_{i=1}^{s-1}f_{ca_i}$
from which we see that a resolution of $I_X$ can be obtained as direct
sum of the resolutions of $ I_{Y_s}(-\deg f)$ and
$I_Y/(f)$; since both have resolutions of  length $c$ the same is true
for $I_X$ and we are done.
\end{proof}

Our next aim is to compute the Hilbert function of a tower scheme in the case when the defining families consist of linear forms. To do this, if $T$ is a tower set, we define a map $\sigma: T\to\zc$ as follows.  Let $\alpha=(a_1,\ldots,a_c)\in T,$ we set
 $$h_1(\alpha):=\mid\{i\mid i\le a_{1},\,(i,a_{2},\ldots,a_{c})\in T\}\mid,$$
 $$h_j(\alpha):=\mid\{i\mid i\le a_{j},\,T_{(i,a_{j+1},\ldots,a_{c})}\ne\emptyset\}\mid,\,\,\text{for}\,\,2\le j\le c$$ 
 and finally
 $$\sigma(a_1,\ldots,a_c):=(h_1(\alpha),\ldots,h_c(\alpha)).$$
The map $\sigma$ is trivially injective. We set $T^{\#}:=\sigma(T).$
\begin{prp}\label{towleft}
For every tower set $T,$ $T^{\#}$ is a left segment.
\end{prp}
\begin{proof}
Let $\alpha'=(a'_1,\ldots,a'_c)\in T^{\#}$ and $\beta'=(b'_1,\ldots,b'_c)\in\zc,$ such that $\beta'\le\alpha'.$ We have to prove that $\beta'\in T^{\#},$ i.e. we have to find $\beta\in T$ such that $\sigma(\beta)=\beta'.$ Let $\alpha=(a_1,\ldots,a_c)\in T$ be such that $\sigma(\alpha)=\alpha',$ hence $a'_i=h_i(\alpha).$ Since $b'_c\le h_c(\alpha),$ there is a unique element $b_c$ such that $T_{b_c}\ne\emptyset$ and $\mid\{i\mid i\le b_c,\, T_{i}\ne\emptyset\}\mid=b'_c.$ 
Now, since $b'_{c-1}\le h_{c-1}(\alpha),$ we have that $T_{(b'_{c-1},a_c)}\ne\emptyset$ and, since $T$ is a tower set, $T_{(b'_{c-1},b_c)}\ne\emptyset,$ therefore there is a unique element $b_{c-1}$ such that $T_{(b_{c-1},b_c)}\ne\emptyset$ and $\mid\{i\mid i\le b_{c-1},\, T_{(i,b_c)}\ne\emptyset\}\mid=b'_{c-1}.$ 
By iterating the same argument we will set $b_j$ the unique element such that $T_{(b_j,b_{j+1},\ldots,b_c)}\ne\emptyset$ and
$\mid\{i\mid i\le b_{j},\, T_{(i,b_{j+1},\ldots,b_c)}\ne\emptyset\}\mid=b'_j,$ for $1\le j\le c.$ Now we set $\beta=(b_1,\ldots,b_c).$ By definition $\beta\in T$ and $\sigma(\beta)=\beta'.$
\end{proof}

\begin{rem}\label{chaintw}
Note that if $T,U\subset\zc$ are tower sets such that $T\subseteq U$ then $T^{\#}\subseteq U^{\#}.$
\end{rem}

\begin{prp}\label{hftower}
Let $T\subset\zc$ be a tower set. Let $X=V(I_T(\ff_1,\ldots,\ff_c)).$ Let $Y$ be a tower scheme supported on $T^{\#},$ with respect to the same families $\ff_1,\ldots,\ff_c.$ Then $H_X=H_Y.$
\end{prp}
\begin{proof}
%Since $X$ and $Y$ are aCM schemes, of the same codimension $c,$ it is enough to prove the assertion for their Artinian reductions. 
%So we can suppose that $R=k[x_1,\ldots,x_c]$ and $R/I_X$ and $R/I_Y$ are Artinian rings. 
%We proceed by induction on $c.$
%\par
If $c=1$ then $T$ is a finite subset of $\zz,$  say $r=|T|,$ so $T^{\#}=[r].$ Therefore $I_X$ and $I_Y$ are principal ideals generated by a form of same degree, hence $H_X=H_Y.$
\par
So we may assume that $c\ge 2$ and we proceed by induction on $c.$
We consider the set $\pi_c(T)=\{m_1,\ldots,m_s\},$ $m_1<\ldots<m_s.$ Since $T$ is a tower set, $T_{m_1}\supseteq\ldots\supseteq T_{m_s}$ are $(c-1)$-tower sets. Let $X_i$ be the scheme defined by $I_{T_{m_i}}(\ff_1,\ldots,\ff_{c-1}).$ Then each $X_i$ is an aCM scheme of codimension $c-1$ by Theorem \ref{tower-aCM} and $X_1\supseteq\ldots\supseteq X_s.$ Moreover, by Remark \ref{chaintw}, 
$(T_{m_1})^{\#}\supseteq\ldots\supseteq (T_{m_s})^{\#}.$ Let $Y_i$ be the scheme defined by $I_{(T_{m_i})^{\#}}(\ff_1,\ldots,\ff_{c-1}).$ By the inductive hypothesis $H_{X_i}=H_{Y_i}.$ Now, let $\ff_c=\{f_1,\ldots,f_s\},$ we set $\overline{X}_i:=X_i\cap V(f_i)$ and $\overline{Y}_i=Y_i\cap V(f_i).$ Since $X_i$ and $Y_i$ are aCM then $H_{\overline{X}_i}=H_{\overline{Y}_i}.$ Finally, using induction on $s$ and the exact sequences (see Lemma \ref{l1})
  $$ 0 \to   I_{\overline{X}_s}(-\delta)  \to  I_X  \to I_{\overline{X}_i\cup\ldots\cup\overline{X}_{s-1}}/(f_1\ldots f_{s-1})  \to   0 $$
  $$ 0 \to   I_{\overline{Y}_s}(-\delta)  \to  I_Y  \to I_{\overline{Y}_i\cup\ldots\cup\overline{Y}_{s-1}}/(f_1\ldots f_{s-1})  \to   0 $$
where $\delta=\deg(f_1\ldots f_{s-1}),$ we get the conclusion.
\end{proof}

Proposition~\ref{hftower} allows us to find a formula for the Hilbert function of a tower scheme.

\begin{rem}\label{dip}
Note that, according to the exact sequence of Lemma \ref{l1}, the Hilbert function of a tower scheme depends on the tower set and on the degrees of the forms in the families.
\end{rem}

Now we associate to a tower scheme $X,$ supported on a left segment $L,$ a partial intersection $Y$ with support on a suitable left segment $L_D$ such that $H_X=H_Y.$ 

If $L$  is a left segment of size $(a_1,\ldots,a_c)$ and $D=\{d_{ij}\},$ $1\le i\le c$ and $1\le j\le a_i$ are positive integers, we define a new left segment,  which we will be denoted by $L_D,$ in the following way. If $L$ is (minimally) generated by $K_1,\ldots,K_r,$ then $L_D$ is the left segment generated by $K'_1,\ldots,K'_r$ where, if $K_i=(k_1,\ldots,k_c)$ then $K'_i=(\sum_{j=1}^{k_1}d_{1j},\ldots,\sum_{j=1}^{k_c}d_{cj}).$

Thus, let $X$ be a tower scheme supported on the left segment $L$ and let $\ff_i=\{f_{ij}\mid j\in\pi_i(L)\},$ for $1\le i \le c,$ be the families defining $X.$ Set now 
$\pi_c(L)=\{a_1,\ldots,a_c\},$ with $a_1<\ldots<a_c$ and $D=\{d_{ij}\}$ where $d_{ij}=\deg f_{ij},$ $1\le i\le c$ and $1\le j\le a_i.$ Since, by Remark \ref{dip},  $H_X$ depends just on $D$ we may assume that $f_{ij}=\prod_{h=1}^{d_{ij}}l_{ij}^h$, where each $l_{ij}^h$ is a linear form. Now we denote by $Y$ the partial intersection supported on $L_D$ with respect the $c$ ordered families of linear forms  %$\mathcal{L}_i=\{l_{ij}^h \ | \ 1\le j\le a_i, 1\le h \le d_{ij}\}.$ 
$$\mathcal{L}_i=(l_{i11},\ldots,l_{i1d_{i1}},l_{i21},\ldots,l_{i2d_{i2}},\ldots,l_{ia_i1},\ldots,l_{ia_id_{ia_i}}).$$

\begin{prp}\label{LD}
Given a tower scheme $X$ supported on the left segment $L$ with respect to the families of forms  $\ff_i=\{f_{ij}\mid j\in\pi_i(L)\},$ for $1\le i \le c$  and $D=\{d_{ij}\}$ where $d_{ij}=\deg f_{ij}.$ 
%Let $Y$ be the partial intesection supported on $L_D$ as defined above. 
Then $H_X=H_{L_D}.$ 
\end{prp}
\begin{proof}
By definition $I_X=\cap_{(j_1, \ldots,j_c)\in L}(f_{1j_1}, \ldots,f_{cj_c}).$ 
Now we denote by $Y$ the partial intersection supported on $L_D$ with respect the $c$ ordered families of linear forms  %$\mathcal{L}_i=\{l_{ij}^h \ | \ 1\le j\le a_i, 1\le h \le d_{ij}\}.$ 
$\mathcal{L}_i=(l_{i11},\ldots,l_{i1d_{i1}},l_{i21},\ldots,l_{i2d_{i2}},\ldots,l_{ia_i1},\ldots,l_{ia_id_{ia_i}}),$ $1\le i\le c.$ Now if $\alpha$ is an integer such that 
$1\le\alpha\le\sum_{s=1}^{a_i}d_{is}$ we set $$t_{\alpha}:=\max\{j\mid d_{i1}+\ldots+d_{ij} <\alpha\}+1$$ and
 $$h_{\alpha}:=\alpha-\sum_{s=1}^{t_{\alpha}-1}d_{is}$$ and $p_{i\alpha}:=l_{it_{\alpha}h_{\alpha}}.$ Then, with this notation 
 $$I_Y=\bigcap_{(\alpha_1,\ldots,\alpha_c)\in L_D}(p_{1\alpha_1},\ldots,p_{c\alpha_c}).$$
It is a matter of computation to show that $I_X=I_Y$ and this completes the proof.
\end{proof}

In the next corollary we lead back the computation of the Hilbert function of a tower scheme to that of a partial intersection. The Hilbert function of a 
partial intersection was explicitly computed in \cite{RZ1}.

\begin{cor}\label{HFtow}
If $X$ is a tower scheme supported on a tower set $T$ with respect to families of forms of degrees $D,$ then $H_X=H_{(T^{\#})_D}.$
\end{cor}
\begin{proof}
It follows just using Propositions \ref{hftower} and \ref{LD}.
\end{proof}

\section{Generalized tower sets: a characterization of aCM property} 
\markboth{\it Generalized tower sets: a characterization of aCM property}{\it Generalized tower sets: a characterization of aCM property}
In this section we will generalize tower sets in such a way to characterize aCM squarefree monomial ideal of codimension $2.$
%we will study equidimensional squarefree monomial ideals arising from tower sets. 
\par
Let $I\subset k[x_1,\ldots,x_n]$ be an equidimensional squarefree monomial ideal of codimension $c$ and let $I=\prm_1\cap\ldots\cap\prm_t$ be its primary decomposition. Each $\prm_i$ is a prime ideal of the type $(x_{a_{i1}},\ldots,x_{a_{ic}}).$ So we can consider the subset  
$\sd(I):=\{\{a_{i1},\ldots,a_{ic}\}\mid 1\le i\le t\}$ of $C_{c,n}.$ 
%\par
%Let $\sigma:\zc\to C_{c,n}$ be the map defined by $\sigma(a_1,\ldots,a_n)=\{a_1,\ldots,a_n\}.$
\par
Vice versa to $\sd\subseteq C_{c,n}$ we can associate an equidimensional squarefree monomial ideal 
$$I_{\sd}:=\bigcap_{\{a_1,\ldots,a_c\}\in\sd} (x_{a_1},\ldots,x_{a_c}).$$
\begin{dfn}
Let $\sd\subseteq C_{c,n}.$ We will say $\sd$ aCM if $I_{\sd}$ is an aCM ideal.
\end{dfn}

\begin{dfn}
Let $\sd\subseteq C_{c,n}.$ We will say that $\sd$ is {\em towerizable} if there exists a permutation $\tau$ on $[n]$ and an ordinante function 
$\omega:C_{c,n}\to (\zc)^*$ such that $\tau(\omega(\sd))$ is a tower set.
\end{dfn}
%\begin{rem}
%Collegare $I_{\sd}$ con un tower scheme associato a $\omega(\sigma(\sd)).$
%\end{rem}
\begin{rem}
Let $\sd\subseteq C_{c,n}.$ Note that $\sd$ is {\em towerizable} if there exists a tower set $T$ and families $\ff_i\subseteq \{x_1,\ldots,x_n\},$ such that  $\sd(I_T(\ff_1,\ldots,\ff_c))=\sd.$
\end{rem}
%Vice versa to $\sd\subseteq C_{c,n}$ we can associate an equidimensional squarefree monomial ideal 
%$$I_{\sd}=\bigcap_{\{a_1,\ldots,a_c\}\in\sd} (x_{a_1},\ldots,x_{a_c}).$$
%\begin{dfn}
%Let $\sd\subseteq C_{c,n}.$ We will say $\sd$ aCM if $I_{\sd}$ is an aCM ideal.
%\end{dfn}

By Theorem \ref{tower-aCM} if $\sd$ is towerizable then $\sd$ is aCM, however there are aCM equidimensional squarefree monomial ideals $I$ such that $\sd(I)$ is not towerizable. Here it is an example in codimension $2.$

\begin{exm}\label{ex}
Let $\sd=\big\{\{1,2\},\{3,4\},\{5,6\},\{4,6\},\{1,4\},\{1,6\}\big\}.$ Then it is easy to check that $I_{\sd}$ is the determinantal ideal generated by the order $3$ minors of the matrix
$$\begin{pmatrix}x_1 & 0 & 0 \\ x_2 & x_3 & x_5 \\ 0 & x_4 & 0 \\ 0 & 0 & x_6\end{pmatrix},$$
so $\sd$ is aCM. Let us suppose that $\sd$ is towerizable. Then there exists a tower scheme $X$ with support on a tower set $T$ such that $\sd(I_X)=\sd.$ Of course $|T|=6$ and there is not a variable $x_k$ such that the ideal $(x_k)$ contains $4$ of the $6$ minimal primes of $I_{\sd}.$ Consequently, $|\pi_2(T)|\le 3$ and for every $a\in\pi_2(T)$ $|T_a|\le 3$ so we have only three possibilities 
\begin{itemize}
	\item[1)] $\pi_2(T)=\{a,b\}$ with $|T_a|=3$ and $|T_b|=3;$
	\item[2)] $\pi_2(T)=\{a,b,c\}$ with $|T_a|=2,$ $|T_b|=2$ and $|T_c|=2;$
	\item[3)] $\pi_2(T)=\{a,b,c\}$ with $|T_a|=3,$ $|T_b|=2$ and $|T_c|=1.$
\end{itemize}
The first two cases cannot occur since $I_X=I_S$ does not contain monomials of degree two.
\par
Therefore $T_a\supset T_b\supset T_c$ and $T_a=\{h_1,h_2,h_3\},$ $T_b=\{h_1,h_2\},$ $T_c=\{h_1\}$ for some $h_i$'s and thus $T=\big\{(h_1,a),(h_2,a),(h_3,a),(h_1,b),(h_2,b),(h_1,c)\big\}.$ But the numbers $2,$ $3$ and $5$ belong each to one only element of $S$ whereas in $T$ there are only two such numbers, precisely $h_3$ and $c.$ 
\end{exm}
Because of the previous example it is natural to ask which sets $\sd\subseteq{C_{c,n}}$ are aCM. We will give a characterization in codimension $2$ 
(see Theorems \ref{gtsacm} and \ref{gts-acm}).

\begin{dfn}\label{div}
Let $h\in\zp.$ Let $\sd\subseteq C_{2,n}.$ We set $$\sd:h:=\{A\in\sd\mid h\not\in A\}.$$ If $S\subseteq\zp^2$ 
we set $$S:h:=\{\alpha\in S\mid\pi_1(\alpha)\ne h\text{ and }\pi_2(\alpha)\ne h\}.$$
\end{dfn}

\begin{rem}\label{div-acm}
Note that if $\sd$ is aCM then $\sd:h$ is aCM. Indeed, if $M$ is an Hilbert-Burch matrix for $I_S$ then $I_{\sd:h}$ is generated by the maximal minors of the matrix obtained from $M$ by replacing $x_h$ with $1.$
\end{rem}

In the sequel we will use the following result which shows that if $S\subseteq{C_{2,n}}$ is aCM then also the scheme obtained by replacing 
$(x_i,x_j)\supseteq I_{\sd}$ with $(h_i,h_j),$ generic complete intersections, is aCM.

\begin{prp}\label{acmci}
Let $\sd\subseteq{C_{2,n}}$ be an aCM set, with $I_{\sd}\subset k[x_1,\ldots,x_n].$ Let $h_1,\ldots,h_n\in k[y_1,\ldots,y_m]$ be forms, such that $\depth(h_i,h_j)=2$ for every $\{i,j\}\in\sd$ and
$$\depth(h_i,h_j,h_u,h_v)\ge 3$$ for every $\{i,j\},\{u,v\}\in\sd,$ $\{i,j\}\ne\{u,v\}.$
Then the ideal $J=\bigcap_{\{i,j\}\in\sd}(h_i,h_j)$ is aCM.
\end{prp}
\begin{proof}
We consider the following vectors $$\xu=(x_1,\ldots,x_n) \text{ and } \hu=(h_1,\ldots,h_n).$$
\par
Since $I_{\sd}$ is aCM we can consider $M=M(\xu),$ an Hilbert-Burch matrix for $I_{\sd}.$ We claim that $N=M(\hu)$ is an Hilbert-Burch matrix for $J.$ We have to prove that $J=I(N).$ Let $g\in I(N)$ be a maximal minor of $N.$ Then $g=f(\hu),$ with $f(\xu)\in I_S.$ Therefore 
$f(\xu)=\lambda_i(\xu)x_i+\mu_j(\xu)x_j$ for every $\{i,j\}\in S,$ consequently $f(\hu)\in (h_i,h_j)$ for every $\{i,j\}\in\sd.$ So $I(N)\subseteq J.$ 
\par
To conclude the proof it is enough to show that $\deg I(N)=\deg J.$ By the generality of the forms $h_1,\ldots,h_n,$ we have that 
 $$\deg J=\sum_{\{i,j\}\in\sd}(\deg h_i)(\deg h_j).$$
Now we proceed by induction on $n.$ If $n=2$ then $\sd=\{1,2\}$ and $I(N)=(h_1,h_2).$ So we can suppose that 
$\deg I(N)=\sum_{\{i,j\}\in\sd}(\deg h_i)(\deg h_j),$ when $\sd\subseteq C_{2,n-1}.$ We can write $\sd=(\sd:n)\cup{\sd}_{(n)}$ where
$\sd:n=\{\alpha\in S\mid n\not\in\alpha\}$ and $\sd_{(n)}=\sd\setminus (\sd:n).$ By Remark \ref{div-acm}, $\sd:n$ is aCM.
%and $I_{S:n}=I_S:x_n.$
Let $\widehat{M}(\xu)$ be an Hilbert-Burch matrix for $I_{\sd:n}.$ We set $\widehat{N}=\widehat{M}(\hu).$ 
%Then $\deg I(\widehat{N})=\sum_{\{i,j\}\in (S:n)}(\deg h_i)(\deg h_j).$ 
%Moreover $I_S=I_{S:n}\cap (x_n,\prod_{u\in S_{(n)}}{x_u}).$
Therefore 
 $$I_{\sd}=I(M)=(I(M):x_n)\cap (x_n,\prod_{u\in {\sd}_{(n)}}{x_u})=I(\widehat{M})\cap (x_n,\prod_{u\in{\sd}_{(n)}}{x_u}).$$
Hence, using the inductive hypothesis, we get
 $$\deg I(N)=\deg (I(\widehat{N}))+\deg (h_n,\prod_{u\in {\sd}_{(n)}}{h_u}))=$$
 %$$J=\bigcap_{\{i,j\}\in S}(h_i,h_j)=\bigcap_{\{i,j\}\in (S:n)}(h_i,h_j)\cap\bigcap_{\{i,j\}\in S_{(n)}}(h_i,h_j)=$$
 %$$\bigcap_{\{i,j\}\in (S:n)}(h_i,h_j)\cap\bigcap_{\{u,n\}\in S_{(n)}}(h_u,h_n)=
 %\bigcap_{\{i,j\}\in (S:n)}(h_i,h_j)\cap(h_n,\prod_{\{u,n\}\in S_{(n)}} h_u)\rw$$
 $$\sum_{\{i,j\}\in (\sd:n)}(\deg h_i)(\deg h_j)+(\deg h_n)\sum_{\{u,n\}\in {\sd}_{(n)}} \deg h_u=
 \sum_{\{i,j\}\in\sd}(\deg h_i)(\deg h_j).$$
%Now let $g(\yu)\in J.$ Then $g(\yu)=g_i(\yu)h_i+g_j(\yu)h_j,$ for every $\{i,j\}\in S.$ Let $m_1,\ldots,m_r$ be the maximal %minors of $M(\xu).$ Then $I_S=(m_1,\ldots,m_r).$ Now let us consider the ideal $\widetilde{I_S}=(m_1,\ldots,m_r)\subset %k[\xu,\yu].$ Of course $\widetilde{I_S}$ is aCM and $M(\xu)$ is an Hilbert-Burch matrix for $\widetilde{I_S}.$ We set %$\phi_{ij}(\xu,\yu)=g_i(\yu)x_i+g_j(\yu)x_j.$
%On the other hand, since $I(M)$ is a monomial ideal, $I(M)$ has a monomial basis $\{m_1,\ldots,m_r\}.$ Therefore $I(M)$ has as a %minimal set of generatos $\{m_1(H_1,\ldots,H_n),\ldots,m_r(H_1,\ldots,H_n)\}.$ So $I(N)=\bigcap_{\{i,j\}\in V}(H_i,H_j),$ with %$S\subseteq V.$
\end{proof}

%In the sequel we set $\Delta=\{(x,y)\in\zp^2\mid x=y\}.$
We recall that if $T\subseteq\zp^2$ and $i\in\zp$ then $$T^i=\{j\in\pi_2(T)\mid (i,j)\in T\}$$ and $$T_i=\{j\in\pi_1(T)\mid (j,i)\in T\}.$$

\begin{rem}\label{remphi}
Let $T\subset\zd$ be a tower set. Then
\begin{itemize}
	\item [1)] $a<b$ and $T_a\ne\emptyset$ $\rw$ $(a,b)\not\in T.$ Indeed, the assumption implies 
	$T_a\supseteq T_b.$ Since $(a,a)\not\in T,$ we have $a\not\in T_a,$ therefore $a\not\in T_b$ i.e. $(a,b)\not\in T.$
	\item [2)] $a<b$ and $(b,a)\in T$ $\rw$ $(a,b)\not\in T.$ Indeed, the assumption implies $T_a\ne\emptyset$ so, by the previous item, $(a,b)\not\in T.$
	\item [3)] $\{(a,b),(b,a)\}\not\subseteq T$ for every $a$ and $b.$ It follows by item $2.$
\end{itemize}
Note that by item $3,$ $|T|=|\varphi(T)|$ where $\varphi$ is the forgetful function.
\end{rem} 
%\begin{proof}
%\begin{itemize}
%	\item [1)] If $a<b$ and $T_a\ne\emptyset$ then $T_a\supseteq T_b.$ Since $(a,a)\not\in T,$ we have $a\not\in T_a,$ therefore $a\not\in T_b$ i.e. $(a,b)\not\in %T.$
%	\item [2)] If $a<b$ and $(b,a)\in T$ then $T_a\ne\emptyset$ so, by the previous item, $(a,b)\not\in T.$
%	\item [3)] It follows by item $2.$
%\end{itemize}
%\end{proof}

\begin{prp}\label{righe}
Let $T\subset\zp^2$ be a tower set. Then $T^i$ and $T^h$ are comparable under inclusion for every $i$ and $h.$
\end{prp}
\begin{proof}
Let $j\in T^i$ be such that $j\not\in T^h,$ we have to show that $T^h\subset T^i.$ Let $k\in T^h,$ i.e. $h\in T_k;$ but $h\not\in T_j$ therefore since $T$ is a tower set we have that $T_j\subset T_k,$ so $i\in T_k$ i.e. $(i,k)\in T$ that implies that $k\in T^i.$
\end{proof}

\begin{prp}\label{nonrip}
Let $T\subset\zd$ be a tower set. 
\begin{itemize}
	\item [1)] Let $h\in\pi_2(T)$ be such that $T_h\supseteq T_j$ for every $j\in\pi_2(T).$ Then $h\not\in\pi_1(T).$
	\item [2)] Let $h\in\pi_1(T)$ be such that $T^h\supseteq T^i$ for every $i\in\pi_1(T).$ Then $h\not\in\pi_2(T).$
\end{itemize}

\end{prp}
\begin{proof}
\begin{itemize}
	\item [1)] If $(h,j)\in T$ then $h\in T_j\subseteq T_h,$ i.e. $(h,h)\in T.$
	\item [2)] Using Proposition \ref{righe} the proof is analogous to item $1.$
\end{itemize}

\end{proof}

Let $T\subset\zd$ be a tower set. Let $h\in\pip.$ We set
% $$F_T(h)=\{j\in\pi_2(T)\mid T_h\subset T_j \text{ and }j>k\,\,\forall (h,k)\in T\}.$$
  $$F_T(h):=\{j\in\pi_2(T)\mid T_h\subset T_j \text{ and } (h,j)\not\in T\}.$$
Note that if $j\in F_T(h)$ then $j<h.$
%\par
%Let $\pip=\{c_1,\ldots,c_r\},$ with $c_1<\ldots<c_r.$ We set 
%  $$C_T(c_j)=
%\left\{
%\begin{array}{ll}
%C_T(c_{j-1})\cup\{c_j\} & \text{if }  F_T(c_{j})\subseteq C_T(c_{j-1}) \\
%C_T(c_{j-1}) & \text{otherwise}
%\end{array} \right.
%$$
%for $1\le j\le r,$
% $$C_T(c_j)=\{h\in\pip\mid F_T(h)\subseteq C_T(c_{j-1})\}\cup C_T(c_{j-1}),$$
%where we set $C_T(c_0)=\emptyset.$
%Finally we set $C_T=C_T(c_r).$
%\par
%Note that if $j\in C_T$ then $F_T(j)\subseteq C_T.$
 %$$C_T=\{h\in\pip\mid F_T(h)\subseteq\pip\}.$$
 
\begin{dfn}
Let $U\subseteq C_{2,n}.$ We say that $U$ is {\em connected} if for every $A,B\in U$ there is $C\in U$ such that $A\cap C\ne\emptyset$ and $B\cap C\ne\emptyset.$ Let $S\subset\zd.$ We say that $S$ is {\em connected} if $\varphi(S)$ is connected.
\end{dfn}

\begin{dfn}\label{tg}
Let $S\subset\zd$ be a finite set. We say that $S$ is a {\em generalized tower set} if 
\begin{itemize}
	\item [1)] $S$ is connected;
	\item [2)] $S=T\cup S_0$ where $T$ is a tower set
\end{itemize}
and $S_0$ has the following further properties
\begin{itemize}
	\item [3)] for every $(i,j)\in S_0,$ $i\not\in\pi_1(T)\cup\pi_2(T)$ and $j\in\pi_1(T)\cap\pi_2(T);$
	\item [4)] for every $(i,j)\in S_0$ and $h\in F_T(j),$ $(i,h)\in S_0.$
\end{itemize}
\end{dfn}

\begin{dfn}\label{gts}
Let $S\subset({\zz}_+^2)^*$ be a generalized tower set. Let $R=k[x_1,\ldots,x_n],$ $n\ge 3.$ Let $\ff_i=\{f_{ij}\mid j\in\pi_i(S)\},$ $1\le i\le 2,$ where each $f_{ij}$ is a form satisfying such conditions of genericity: for every $(a_1,a_2)\in S,$ $f_{1a_1},f_{2a_2}$ are coprime and for every $(a_1,a_2),$ $(b_1,b_2)\in S,$ with $\{a_1,a_2\}\ne\{b_1,b_2\},$ $\depth(f_{1a_1},f_{2a_2},f_{1b_1},f_{2b_2})\ge 3.$
%$f_{1j}$ and $f_{1h}$ are coprime when there is $k$ such that $(j,k)$ and $(h,k)$ belong to $S$ and for every $\alpha=(a_1,a_2)\in S$ the forms %$(f_{1a_1},f_{2a_2})$ are coprime. 
If $\alpha=(a_1,a_2)\in S,$ we will denote by $I_{\alpha}$ the complete intersection ideal generated by $f_{1a_1},f_{2a_2}.$ We set
 $$I_S(\ff_1,\ff_2):=\bigcap_{\alpha\in S}I_{\alpha}.$$
It defines a $2$-codimensional subscheme of $\pp^n$ called {\em generalized tower scheme}, with support on $S,$ with respect to the families $\ff_1,\ff_2.$
\end{dfn}

In the sequel if $S\subset({\zz}_+^2)^*$ we will set for short $I_S:=I_{\varphi(S)},$ consequently $S$ will be said aCM if $I_S$ is aCM.
\par
In order to prove our results on the Cohen-Macaulayness of such schemes we need several lemmas.

\begin{lem}\label{minrg}
Let $S=T\cup S_0$ be a generalized tower set. Let $i\in\pi_1(S_0)$ and let $m=\min\{j\mid (i,j)\in S_0\}.$
Then $F_T(m)=\emptyset.$
\end{lem}
\begin{proof}
Let $s\in F_T(m);$ then $s<m$ and by Definition \ref{tg}, item 4, $(i,s)\in S_0,$ which is a contradiction.
\end{proof}

%\begin{dfn}\label{div}
%Let $h\in\zp.$ Let $U\subseteq C_{2,n}.$ We set $$U:h=\{A\in U\mid h\not\in A\}.$$ If $S\subseteq\zp^2$ 
%we set $$S:h=\{\alpha\in S\mid\pi_1(\alpha)\ne h\text{ and }\pi_2(\alpha)\ne h\}.$$
%\end{dfn}

\begin{lem}\label{ctdiv}
With the above notation, if $h\in\pip$ then for every $j\in\pip\setminus\{h\}$ we have $F_{T:h}(j)\subseteq F_{T}(j).$
%\begin{itemize}
	%\item [1)] For every $j\in C_T\setminus\{h\}$ we have $F_{T:h}(j)\subseteq F_{T}(j);$
	%\item [2)] $C_T\setminus\{h\}\subseteq C_{T:h}.$
%\end{itemize}
\end{lem}
\begin{proof}
%\begin{itemize}
%	\item [1)] 
If $b\in F_{T:h}(j)$ then $(T:h)_j\subset (T:h)_b$ and $(j,b)\not\in T:h,$ with $j\ne h$ and $b\ne h,$ so $(j,b)\not\in T.$ Moreover there is $a$ such that $(a,b)\in T:h$ and $(a,j)\not\in T:h.$ Since $a\ne h$ this implies that $T_b\not\subseteq T_j.$ Since $T$ is a tower set we get that $T_j\subset T_b.$ 
%	\item [2)]
%	Let $C_T=\{c_1,\ldots,c_p\},$ with $c_1<\ldots<c_p.$ 
%	If $c_1\in C_T\setminus\{h\}$ then, by item $1,$ $F_{T:h}(c_1)\subseteq F_{T}(c_1)=\emptyset,$ hence $c_1\in C_{T:h},$ i.e. $\{c_1\}\setminus\{h\}\subseteq %C_{T:h}.$ Now, by recursion, let us suppose that $\{c_1,\ldots,c_{j-1}\}\setminus\{h\}\subseteq C_{T:h}.$
%	If $c_j\in C_T\setminus\{h\}$ then, by item $1,$ $F_{T:h}(c_j)\subseteq F_{T}(c_j)\setminus\{h\}\subseteq %C_T(c_{j-1})\setminus\{h\}=\{c_1,\ldots,c_{j-1}\}\setminus\{h\}\subseteq C_{T:h},$ hence $\{c_1,\ldots,c_{j}\}\setminus\{h\}\subseteq C_{T:h},$ 
%	for $1\le j\le p.$
%\end{itemize}
 
\end{proof}

\begin{lem}\label{divtg}
Let $S=T\cup S_0$ be a generalized tower set. Let $h\in\pip.$ Then $S:h$ is a generalized tower set with respect to the
decomposition $S:h=(T:h)\cup(S_0:h).$
\end{lem}
\begin{proof}
Of course $S:h=(T:h)\cup(S_0:h).$ 
\begin{itemize}
	\item [1)] Since $S$ is connected then $S:h$ is connected too.
	\item [2)] Let $a,b\in\pi_2(T:h),$ $a<b.$ Let $i\in(T:h)_b;$ then $(i,b)\in T:h$ i.e. $i\in T_b\subseteq T_a;$ 
	since $i\ne h$ and $b\ne h$ then $i\in(T:h)_a.$
	\item [3)] Let $(i,j)\in S_0:h.$ Of course $i\not\in\pi_1(T:h)\cup\pi_2(T:h).$
	Moreover $(\pip)\setminus\{h\}=\pi_1(T:h)\cap\pi_2(T:h).$ Indeed, $\pi_1(T:h)\cap\pi_2(T:h)\subseteq(\pip)\setminus\{h\}$ trivially;
	if $k\in(\pip)\setminus\{h\}$ then $k\ne h,$ $k\in T_b$ where $T_b\supseteq T_u$ for every $u$ and $k\in T^a$ where $T^a\supseteq T^v$ for every $v,$ 
	(see Proposition \ref{righe}), so by Proposition \ref{nonrip}, $a\ne h$ and $b\ne h$ i.e. $(a,k),(k,b)\in T:h,$ hence item 3 is clear.
	%Let $(i,j)\in S_0:h.$ Of course $i\not\in\pi_1(T:h)\cup\pi_2(T:h),$ hence item 3 is done.
	%Since $j\in C_T\setminus\{h\},$ using Lemma \ref{ctdiv}, $j\in C_{T:h}.$
	%Moreover  this implies that $C_T\setminus\{h\}\subseteq C_{T:h},$ hence item 3 is done.
	%let $(i,j)\in S_0:h$ i.e. $j\in C_T\setminus\{h\}\subseteq C_{T:h}.$
	%and consequently $C_{T:h}=C_T\setminus\{h\}.$
	\item [4)] Let $(i,j)\in S_0:h$ and $u\in F_{T:h}(j).$ By Lemma \ref{ctdiv}, $u\in F_T(j),$ so $(i,u)\in S_0.$ Since $i\ne h$ and $u\ne h$ we get 
	$(i,u)\in S_0:h.$
\end{itemize}
\end{proof}

\begin{lem}\label{divm}
Let $U\subseteq C_{2,n}$ and $a\in [n].$	Then $I_{U:a}=I_U:(x_a).$
\end{lem}
\begin{proof}
$I_U=\bigcap_{\{i,j\}\in U}(x_i,x_j).$ Then
 $$I_U:(x_a)=\bigcap_{\{i,j\}\in U}\big((x_i,x_j):(x_a)\big)=\bigcap_{\{i,j\}\in U:a}(x_i,x_j)=I_{U:a}.$$
\end{proof}

\begin{lem}\label{vci}
Let $S=T\cup S_0$ be a generalized tower set. For every $a\in\pi_1(S_0),$ there is $(a,h)\in S_0$ such that $I_{S:a}+(x_h)$ is a complete intersection ideal of height $2.$
\end{lem}
\begin{proof}
%\begin{itemize} 
We proceed step by step.
\begin{itemize}
	\item[1)] In this first step we show that the assertion is equivalent to prove that for all $\{i,j\}\in \varphi(S:a)$ we have either $\{i,h\}\in \varphi(S:a)$ or $\{j,h\}\in \varphi(S:a).$
	\par
At the beginning we observe that if $I_{S:a}+(x_h)$ is equidimensional of height $2$ and $\prm$ is a minimal prime in its primary decomposition then $x_h\in\prm,$ so $I_{S:a}+(x_h)=\bigcap_i (x_h,x_i)=(x_h,\prod_i x_i)$ that is a complete intersection.
	\par
	On the other hand to show that $I_{S:a}+(x_h)$ is equidimensional of height $2$ it is enough to prove that 
	for all $\{i,j\}\in \varphi(S:a)$ we have either $\{i,h\}\in \varphi(S:a)$ or $\{j,h\}\in \varphi(S:a).$ In fact let $\prm=(x_i,x_j,x_h)$ be a prime ideal containing $I_{S:a}+(x_h),$ so $\{i,j\}\in\varphi(S:a),$ consequently $\{i,h\}\in \varphi(S:a)$ or $\{j,h\}\in \varphi(S:a)$ i.e. $\prm$ contains a prime ideal of height $2$ containing $I_{S:a}+(x_h).$
	\end{itemize}
Now let $a\in\pi_1(S_0),$ and let $m_1=\min S^a;$ we set
$$U(S^a):=\{ m\in S^a \ | \ T_m=T_{m_1} \text{ and }F_T(m)=\emptyset  \}.$$
Note that by Lemma \ref{minrg}, $m_1\in U(S^a).$
\par
We claim that the integer $h$ which we are looking for can be found in $U(S^a).$ In the next two steps we prove properties of $U(S^a)$ for our claim.
\begin{itemize} 
	\item[2)] 
	Let $m, n\in U(S^a)$ then $T^m=T^{n}.$ If $\alpha \in T^m\setminus T^n,$ i.e. $(m, \alpha)\in T$ and $(n, \alpha)\notin T$ we have $T_{\alpha}\supset T_m =T_n.$ Now
since $\alpha \notin F_T(n)=\emptyset$ and $(n, \alpha)\notin T$ we should have $T_{\alpha}\subseteq T_n,$ a contradiction.
	\item[3)]  If $F_T(m)=\emptyset$ and $(\alpha, \beta)\in T$ then either $(\alpha, m)\in T$ or $(m, \beta)\in T.$ Since $\beta\notin F_T(m)=\emptyset$ 
	we get either $T_{\beta}\subseteq T_m,$ hence $(\alpha,m)\in T,$ or $(m,\beta)\in T.$  
\end{itemize}
In the remaining steps we will find the integer $h$ working by induction on $|U(S^a)|.$
\begin{itemize}
\item[4)] If $|U(S^a)|=1$ then $U(S^a)=\{m_1\},$ and we would like to show that for all 
$\{i,j\}\in \varphi(S:a)$ we have either $\{i,m_1\}\in \varphi(S:a)$ or $\{j,m_1\}\in \varphi(S:a).$
If $\{i,j\}\in \varphi(T)$ then by item $3$ we are done. So we can assume that $\{i,j\}\in \varphi(S_0),$ with $(i,j)\in S_0.$ Of course we can suppose that $j\ne m_1.$
%And suppose by contradiction $(\alpha, \beta)\in S_0$ such that $$\{\alpha, m_1\}, \{m_1, \beta\}\notin \varphi(S).$$Since $S$ is connected then taking $(\alpha, \beta)$ and $(a, m_1)$ we have $$ (a,\beta)\in S.$$
Since $j\notin F_T(m_1)$ then either $(m_1,j)\in T$ or $T_j\subseteq T_{m_1}.$ If $(m_1,j)\in T$ we are done; otherwise we can assume that 
$T_j\subseteq T_{m_1}.$

%\begin{itemize}
	%\item[$\bullet$] $T_{\beta}\supseteq T_k$ for some $(m,k)\in T:$ this implies $(m,\beta)\notin S.$
	%\item[$\bullet$] $T_{m_1}\supseteq T_{\beta}:$
%\begin{itemize}
	\item[\bf{-}] if $T_{j}\subset T_{m_1}$ then either $(j,m_1)\in T$ and we are done, or $(j,m_1)\not\in T$ then $m_1\in F_T(j),$ so by item 4 in Definition \ref{tg} we get $(i,m_1)\in S_0$ and we are done again.
%\end{itemize}
	\item[\bf{-}] if $T_{m_1}= T_{j},$ let us suppose that $$\{i,m_1\}, \{m_1,j\}\notin \varphi(S).$$ Since $S$ is connected then, taking $(i,j)$ and $(a, m_1),$ we get $ (a,j)\in S_0.$
	Since $U(S^a)=\{m_1\}$ and $j\neq m_1$ we have $F_T(j)\neq\emptyset.$ Now take $k\in F_T(j);$ then $T_{m_1}=T_{j}\subset T_{k},$ so $k< m_1.$ Since $(a,j)\in S_0$ and $k\in F_T(j)$ by item 4 in Definition \ref{tg}, we have $(a,k) \in S_0,$ and this contradicts the minimality of $m_1.$
\end{itemize}
%\end{itemize}
%\end{itemize}
	 Let now $U(S^a)=\{m_1, \ldots, m_p\},\ p>1.$ 
\begin{itemize}
\item[5)] At first we prove that $U(S^a)\setminus\{m_p\}=U((S:m_p)^a).$
%$=\{m\in (S:m_p)^a\ |\ (T:m_p)_m=(T:m_p)_{m_1} \ \text{and} \ F_{(T:m_p)}(m)=\emptyset \}.$
 The inclusion  $U(S^a)\setminus\{m_p\}\subseteq U((S:m_p)^a)$ follows directly by definition of $U(S^a)$ and by Lemma \ref{ctdiv}.
	 Let $m\in U((S:m_p)^a)$ then $m\ne m_p$ and $(T:m_p)_m=(T:m_p)_{m_1};$ since $T_{m_1}=T_{m_p}$ we get $m_p\not\in T_{m_1}$ therefore $T_m\supseteq T_{m_1}.$
	 By the minimality of $m_1$ we have $T_m=T_{m_1}.$
	 Now if $b\in F_T(m)$ then $T_{m_1}=T_m\subset T_b,$ so $b<m_1.$ Since $S$ is a generalized tower set we get $(a,b)\in S_0$ and this contradicts the minimality of $m_1.$
\item[6)] By the inductive hypothesis there exists $m\in U(S^a)\setminus\{m_p\}$ such that for any $(i,j)\in (S:a):m_p$ we have either 
 $$\{m,i\} \text{ or }\{m,j\}\in\varphi((S:a):m_p).$$
We will prove that either $m_p$ or $m$ is the wanted element. Let us suppose that there exist $(\alpha,\beta),(u, v)\in S$ such that 
 $$\{m,\alpha\},\{m,\beta\}\notin\varphi((S:a)) \text{ and } \{m_p,u\},\{m_p,v\}\notin\varphi((S:a)).$$
Note that $\beta=m_p,$ since otherwise $(\alpha,\beta)\in S:m_p.$ Now since $v\ne m_p$ $(u,v)\in S_0:m_p,$ hence by hypothesis on $m$ it should be either 
$(u,m)\in S_0$ or $\{v,m\}\in\varphi(T);$ but the last assertion is false since, by item $2,$ $T_m=T_{m_p}$ and $T^m=T^{m_p}.$ 
This implies that $(\alpha,m_p)$ and $(u,m)\in S$ and this contradicts the connection of $S.$
	 
%\end{itemize}
\end{itemize}
%Note that it is enough to show that $I_{S:a}+(x_h)$ is equidimensional of height $2.$ Since in this case every minimal prime should be of the type $(x_h,x_i)$ %for some $i\ne h,$ hence $I_{S:a}+(x_h)=(x_h,\prod_i x_i).$ If $\prm\supseteq (I_{S:a}+(x_h))$ is a prime ideal then $\prm$ contains a prime of the type %$(x_i,x_j,x_h)$ with $(i,j)\in S:a.$ If $(i,j)\in T$ then 
\end{proof}

Finally we are ready to prove the announced result.

\begin{thm}\label{gtsacm}
If $S$ is a generalized tower set then $S$ is aCM.
\end{thm}
\begin{proof}
$S=T\cup S_0,$ where $T$ is a tower set and for $S_0$ the properties of Definition \ref{tg} hold. We proceed by induction on $r=|\pi_1(S_0)|.$
If $S_0=\emptyset$ then $S=T$ which is aCM by Theorem \ref{tower-aCM}. Now we can suppose that the assertion is true up to $r-1.$ Take $a\in\pi_1(S_0).$
Note that $S:a=T\cup(S_0:a)$ and $S:a$ is a generalized tower set with respect to this decomposition. Then by inductive hypothesis $S:a$ is aCM. We can write $I_S=I_{S:a}\cap (x_a,f)$ where $f=\prod_{j\in S^a} x_j.$ Using the exact sequence
 $$0\to I_S \to I_{S:a}\oplus (x_a,f) \to I_{S:a}+(x_a,f) \to 0$$
it is enough to show that $\pd (I_{S:a}+(x_a,f))\le 3,$ i.e. $\pd (I_{S:a}+(f))\le 2.$ To do that we use induction on $\deg f=|S^a|.$
If $f=x_h$ then, by Lemma \ref{vci}, $I_{S:a}+(x_h)$ is a complete intersection ideal of height $2$ and we are done. 

If $\deg f=|S^a|>1,$ by Lemma \ref{vci}, there is $h\in S^a$ such that $I_{S:a}+(x_h)$ is a complete intersection ideal of height $2.$
We can write
 $$I_{S:a}+(f)=\big((I_{S:a}:(x_h))+(f_h)\big)\cap \big(I_{S:a}+(x_h)\big),$$
where $f_h=f/x_h.$ In fact $I_{S:a}+(f)\subseteq\big((I_{S:a}:(x_h))+(f_h)\big)\cap \big(I_{S:a}+(x_h)\big)$ trivially.
Let $g\in \big((I_{S:a}:(x_h))+(f_h)\big)\cap \big(I_{S:a}+(x_h)\big)$ be a monomial. If $g\in I_{S:a}$ we are done, otherwise $g\in (x_h).$ If $g\in (f_h)$ then $g\in (f)$ and we are done again. Otherwise $x_hg\in I_{S:a};$ since $I_{S:a}$ is a monomial squarefree ideal, we get $g\in I_{S:a}.$
\par
By Lemma \ref{divm} we have that $I_{S:a}:(x_h)=I_{(S:h):a}.$ Since by Lemma \ref{divtg} $S:h$ is a generalized tower set, observing that 
$f_h=\prod_{j\in (S:h)^a} x_j,$ we can apply the inductive hypothesis to assert that $$\pd\big((I_{S:a}:(x_h))+(f_h)\big)\le 2.$$ We set 
$J:=(I_{S:a}:(x_h))+(f_h).$ Now let us consider the exact sequence
 $$0\to I_{S:a}+(f) \to J\oplus\big(I_{S:a}+(x_h)\big) \to J+(x_h) \to 0.$$
Since $\pd (J+(x_h))\le 3$ and $I_{S:a}+(x_h)$ is a complete intersection ideal of height $2,$ we can conclude that $\pd(I_{S:a}+(f))\le 2.$
\end{proof}

Now we want to give a converse of the previous theorem. More precisely  we want to prove that every monomial squarefree aCM ideal of height two is supported on a suitable generalized tower set. To do this we introduce some preparatory material.

\begin{dfn}\label{tgs}
Let $U\subseteq C_{2,n}.$ We will say that $U$ is {\em generalized towerizable set} if there exists an ordinante function $\omega:C_{2,n}\to (\zp^2)^*$ and a permutation $\tau$ on $\pi_2(\omega(U))$
such that $\tau\omega(U)$ is a generalized tower set.
\end{dfn}

%Now we want to prove that if $U\subseteq C_{2,n}$ is aCM, then, $U$ is generalized towerizable.
%Now we want to prove that every monomial squarefree aCM ideal of height two is supported on a suitable generalized tower set.
%\begin{thm}
%\label{gts-ci}
%Let $I\subset R$ a monomial squarefree aCM ideal of height $2.$ Then there exists a generalized tower set $S\subset(\zp^2\setminus\Delta)$ and a family of %monomials $h_1,\ldots,h_t,$ such that $I=\bigcap_{(i,j)\in S}(h_i,h_j).$
%\end{thm}
%\begin{proof}
%\end{proof}
%To do this we introduce some preparatory material.

Let $U\subseteq C_{2,n}$ be an aCM set; then, by the Hilbert-Burch theorem, $I_U$ is a determinantal ideal generated by the maximal minors of a matrix of size $(r+1)\times r.$ 
\begin{lem}\label{mat}
If $I\subset R$ is an aCM monomial ideal of height $2$ then it admits a Hilbert-Burch matrix of the form
\begin{tiny}
$$\left( \begin{array}{ccccccccccc}
M_{0,1} & 0 & \ldots &0&&&&&&\ldots&0\\ 
D_1 & M_{1,2} & \ldots & M_{1,\alpha_1} &0&\ldots &&&&&\ldots \\
0 & D_2 & 0 & 0 &M_{2,\alpha_1+1}&\ldots &M_{2,\alpha_2}&&&&\ldots \\
0 & 0 & D_3 & 0 &0& \ldots&0&M_{3,\alpha_2+1}&\ldots&M_{3,\alpha_3}&\ldots \\
  &  &  & \ddots && &&&&&\\
 &  &  &  && &&&&&\\
 &  &  &  && &&&&&\\ 
  &  &  &  && &&&&&\\ 
  &  &  &  && &&&&\ddots&\\ 
0 &\ldots  &  &  && &&&\cdots&0&D_r\\
\end{array} \right)
$$
\end{tiny}where $D_i$ and $M_{ij}$ are monomials, $D_i\ne 0$ and $M_{ij}\ne 0$
and $D_i$ is in the position $(i,i)$ (we enumerate the rows from $0$ to $r$ and the columns from $1$ to $r$).
%We can construct this matrix from any basis of the module of the first sygyzies, where each element plays only on two monomial generators.
\end{lem}
\begin{proof}
We take the minimal monomial set $G$ of generators for $I,$ then the first syzygy module is minimally generated by a set $\Phi$ of $r$ elements acting each only on two of such generators. Moreover there are at least two generators in $G$ on which only one syzygy acts. Let $f_0$ be one of these generators and let $\phi_1$ be the syzygy acting on $f_0$ and let $f_1$ be the other generator on which acts $\phi_1.$ Now we call $\phi_2,\ldots,\phi_{\alpha_1}$ all the other syzygies in $\Phi$ acting respectively on $f_1$ and $f_2,\ldots,f_{\alpha_1}\in G.$ By iterating this procedure we get our matrix.
\end{proof}

\begin{dfn}\label{dfst}
An Hilbert-Burch matrix of the type as in Lemma \ref{mat} will be called a matrix of  {\em standard form}.
\end{dfn}

\begin{dfn}
Let $\mb=(m_{ij})$ be an Hilbert-Burch matrix of {\em standard form} of size $(r+1)\times r.$ Let $\sigma:[r] \to \{0,\ldots,r-1\}$ be the application such that $\sigma(j)$ is the only integer less than $j$ such that $m_{\sigma(j)j}\ne 0.$ 
\end{dfn}

From now on we set $M_i$ for $M_{\sigma(i),i}.$

%\begin{rem}\label{HBsq} Let $I$ be a squarefree monomial ideal of height $2$ generated by $\mathcal{G}(I)=\{f_0, \ldots, f_n\}.$ If $HB_I$ is a Hilbert-Burch %matrix for $I$ of size $n+1\times n$ which has monomial entries then $HB_I$ has squarefree monomial entries. It follows just observing that each determinant %obtained  by $HB_I$ removing a row belong to $\mathcal{G}(I).$
%\end{rem}

%We will say  that $HB_I,$ as described in Lemma \ref{mat},  has {\em standard form}. 

\begin{rem}\label{I_{sigma}} Note that $\sigma(1)=0,$ $\sigma(2)=1$ and, for $j>2,$ $\sigma(j)\ge\sigma(j-1)>0.$ %Therefore,  the function $\sigma$ also defines the ideal $I_{\sigma},$ who arises from the correspondent Hilbert-Burch matrix. 
\end{rem}

Given $j\in \{1, \ldots, r\}$ we denote with  $m(j)$ the set
$$m(j):=\{j,\sigma(j),\sigma^2(j),\ldots,\sigma^h(j)\},$$ where $h$ is the only integer such that $\sigma^h(j)=1.$ 
We write $u\not\in m(j)$ to mean $u\in [n]\setminus m(j).$ 
%$$d(j):=\{1, \ldots, r\}\setminus m(j).$$

\begin{rem}\label{mj}
Note that if $i\in m(j)$ then $m(i)\subseteq m(j).$
\end{rem}

We denote by $f_i$ the determinant of the matrix obtained by removing the row $i$ for $0\le i\le r.$ By the Hilbert-Burch theorem we have that
$$\{f_0,\ldots, f_r\}$$
is a minimal set of generators for $I.$
Note that $f_0=D_0\cdots D_r.$  In the following proposition will compute all the other generators.

\begin{prp}\label{HBgen} 
For any $i\in\{1, \ldots, r\},$  with the above notation, we have
$$f_i=\prod_{j\in m(i)} M_j \cdot \prod_{j\not\in m(i)} D_j.$$
%So we can compute $F_r$ just looking at $m(r),$ for any $r\in\{1, \ldots, n\}.$
\end{prp}
\begin{proof} Let $i\in\{1, \ldots, r\},$ and let $H$ be the square matrix given by $\mb$ without the row containing $D_i.$ Since $M_{i}$ is the only entry in the $i$-th column of $H,$  we compute the determinant $f_i$ by using the Laplace expansion along its $i$-th column. Thus $f_i=M_iG_1,$ where $G_1$ is the determinant of  the matrix $H_1$ obtained from $H$ by deleting the row $\sigma(i)$ and the $i$-th column. Note that $M_{\sigma(i)}$ is the only entry in the $\sigma(i)$-th column of $H_1,$ hence $f_i=M_iM_{\sigma(i)}G_2,$ where $G_2$ is the determinant of  the matrix $H_2$ obtained from $H_1$ by deleting the row 
$\sigma^2(i)$ and the $\sigma(i)$-th column.
So, by iterating this computation, we get $f_i=\prod_{j\in m(i)} M_j \cdot G',$ where $G'$ is the determinant of  the matrix $H'$ obtained from $H$ by deleting the rows $\sigma(j)$ and the columns $j,$ for all $j\in m(i)$.
Finally, we observe that $H'$ is an upper triangular matrix, therefore $G'=\prod_{j\not\in m(i)} D_j.$
\end{proof}

%\begin{dfn}
%Let $\mb$ an Hilbert-Burch matrix of {\em standard form} of size $(r+1)\times r$ and let $I(\mb)$ the ideal generated by the maximal minors of $\mb.$ We set %$h_i=D_i,$ for $1\le i\le r$ and $h_{r+i}=M_i,$ for $1\le i\le r.$
% $$U_{\mb}=\{\{i,j\}\in C_{2,2r}\mid I(\mb)\subseteq (h_i,h_j)\}.$$
%\end{dfn}

%\begin{prp}(sbagliata)
%If $U_{\mb}$ is generalized towerizable then $U$ is generalized towerizable too.
%\end{prp}
%\begin{proof}
%Since $U_{\mb}$ is generalized towerizable there exists a permutation $\sigma$ on $\{1,\ldots,n\}$ and an ordinante function 
%$\omega:C_{2,n}\to (\zc)^*$ such that $S_{\mb}=\omega(\sigma(U_{\mb}))$ is a generalized tower set. Therefore we can write $S_{\mb}=T_{\mb}\cup (S_0)_{\mb}.$

%If $(i,j)\in S_{\mb}$ then there are entries in $\mb,$ $H_i$ and $H_j$ such that $I(\mb)\subseteq (H_i,H_j).$ Note that if $i,h\in\pi_1(S_{\mb}),$ $i\ne h,$ %then
%$\gcd(H_i,H_h)=1.$ It depends on the fact that $I(\mb)$ is a squarefree monomial ideal 
%\end{proof}

Now we want to construct a generalized tower set starting from a n Hilbert-Burch matrix of standard form $\mb.$

\begin{dfn}
We define 
 $$U'_{\mb}:=\{\{i,j\}\in C_{2,2r}\mid i<j\le r \text{ and } i\not\in m(j)\}$$
and
 $$U''_{\mb}:=\{\{i,j\}\in C_{2,2r}\mid i\le r<j \text{ and } j-r\in m(i)\}.$$
Finally we set $$U_{\mb}:=U'_{\mb}\cup U''_{\mb}.$$
%We define $$U_{\mb}=\{\{i,j\}\in C_{2,2r}\mid i<j\le r \text{ and } i\not\in m(j)\}\cup$$
%$$\cup\{\{i,j\}\in C_{2,2r}\mid i\le r<j \text{ and } j-r\in m(i)\}.$$
\end{dfn}

\begin{prp}\label{mdcfr}
\begin{itemize}
\item[1)] If $\{u,v\}\in U'_{\mb}$ then either $u\not\in m(i)$ or $v\not\in m(i)$ for every $1\le i\le r.$
\item[2)] If $\{u,v\}\in U''_{\mb},$ with $u<v,$ then either $u\not\in m(i)$ or $v-r\in m(i)$ for every $1\le i\le r.$
\end{itemize}
\end{prp}
\begin{proof}
\begin{itemize}
\item[1)] Let $u,v\in m(i)$ with $u<v.$ Then $v=\sigma^h(i)$ and $u=\sigma^k(i)$ with $h<k.$ Then $u=\sigma^{k-h}\sigma^h(i)=\sigma^{k-h}(v),$ i.e. $u\in m(v),$ hence $\{u,v\}\not\in U'_{\mb}.$
\item[2)] Let $\{u,v\}\in U''_{\mb},$ with $u<v,$ such that $u\in m(i).$ Since $v-r\in m(u)$ we get $v-r\in m(i).$
\end{itemize}
\end{proof}

\begin{prp}\label{conn}
$U_{\mb}$ is connected.
\end{prp}
\begin{proof} 
Let $\{i,j\}, \{u,v\}\in U_{\mb},$ with $i<j,$ $u<v$ and $u<i.$ 
If $u \not\in m(i)$ then $u<i\le r,$ so $\{u,i\}\in U'_{\mb}.$ If $u\in m(i),$ we have to consider two cases.
\par
%\begin{itemize}
If	$v\le r$  then $\{u,v\}\in U'_{\mb},$ hence $u\not\in m(v).$ Since $u\in m(i)$ we get $m(i)\not\subseteq m(v)$ i.e. $i\not\in m(v).$ Moreover, by Proposition \ref{mdcfr} we have $v\not\in m(i)$ and so $\{i,v\}\in U'_{\mb}.$ 
\par
If  $v>r$ then, by Proposition \ref{mdcfr}, we get $v-r\in m(i),$ so $\{i, v\}\in U''_{\mb}.$  
%\end{itemize}
\end{proof}

Now we set ${\mathcal T}:=U'_{\mb}\cup \{\{i,j\}\in U''_{\mb}\mid j-r\in m(r)\}$ and ${\mathcal S}_0:=U_{\mb}\setminus{\mathcal T},$ so
${\mathcal S}_0=\{\{i,j\}\in C_{2,2r}\mid i\le r<j \text{ and }j-r\in m(i)\setminus m(r)\}.$
%For every $\{i,j\}\in U''_{\mb},$ with $i>j,$ we set $\omega(\{i,j\})=(i,j).$
%We set ${S}_0=\{(i,j)\mid \{i,j\}\in{\mathcal S}_0, \text{ and } i>j\}.$
\par
%Let $\{i,j\}\in{\mathcal T}.$ If $\{i,j\}\in U''_{\mb},$ with $i>j.$ Then we set $\omega(\{i,j\})=(i,j).$
For every $i\in[r],$ we set $\mu_i:=\max(m(i)\cap m(r)).$ 
\par
%Now if $\{i,j\}\in U'_{\mb}$ and $\mu_i<\mu_j$ then we set $\omega(\{i,j\})=(i,j).$
%\par
%Now let $\{i,j\}\in U'_{\mb}$ with $\mu:=\mu_i=\mu_j.$ 
We set 
$${\mathcal V}_i^{(1)}:=\{\{u,v\}\in U'_{\mb}\mid \mu_u=\mu_v=\mu_i\}$$
and
$$r_i^{(1)}:=\max\bigcup_{\alpha\in {\mathcal V}_i^{(1)}}{\alpha}\cup\{0\}.$$
%$$r_1=\max\{u\mid \{u,v\}\in U'_{\mb},\,\mu_u=\mu_v=\mu\}.$$
Note that $r_i^{(1)}<r.$
Moreover, we set 
$$\mu^{(1)}_i:=\begin{cases}\max(m(i)\cap m(r_i^{(1)})) & \text{if }r_i^{(1)}>0 \\ 0 & \text{otherwise}\end{cases}.$$
%If  $\mu^{(1)}_i\ne \mu^{(1)}_j$ then we set 
% $$\omega(\{i,j\})=\begin{cases}(i,j) \text{ if } \mu^{(1)}_i< \mu^{(1)}_j \\ (j,i) \text{ if }\mu^{(1)}_i> \mu^{(1)}_j\end{cases}.$$
Now by induction let us suppose that we defined $\mu^{(k)}_i$ for $1\le k<h.$ Then we set 
$${\mathcal V}_i^{(h)}:=\{\{u,v\}\in{\mathcal V}_i^{(h-1)}\mid \mu^{(h-1)}_u=\mu^{(h-1)}_v=\mu^{(h-1)}_i\}$$
and
$$r_i^{(h)}:=\max\bigcup_{\alpha\in {\mathcal V}_i^{(h)}}{\alpha}\cup\{0\}.$$
Note that if $r_i^{(h-1)}>0$ then $r_i^{(h)}<r_i^{(h-1)}.$ If $r_i^{(h-1)}=0$ then $r_i^{(h)}=0.$
Moreover we set 
 $$\mu_i^{(h)}:=\begin{cases}\max(m(i)\cap m(r_i^{(h)}))& \text{if } r_i^{(h)}>0 \\ 0 & \text{otherwise}\end{cases}.$$
In the sequel we will set $\mu_i^{(0)}:=\mu_i.$ 
Note that by definition the sequence $(\mu_i^{(h)})_{h\ge 0}$ vanishes definitively. 

\begin{lem}\label{seq}
If $\{i,j\}\in U'_{\mathcal M}$ then there exists $h$ such that $\mu_i^{(h)}\ne\mu_j^{(h)}.$
\end{lem}
\begin{proof}
If $\mu_i^{(h)}=\mu_j^{(h)}$ for every $h$ then ${\mathcal V}_i^{(h)}={\mathcal V}_j^{(h)}\ne\emptyset$ for every $h.$ 
Consequently $(\mu_i^{(h)})_{h\ge 0}$ should not be definitively null, a contradiction.
\end{proof}

For every $\{i,j\}\in U''_{\mb},$ with $i>j,$ we set $\omega(\{i,j\}):=(i,j).$ 
\par
For every $\{i,j\}\in U'_{\mb},$ let $t$ be the smallest integer such that $\mu^{(t)}_i\ne \mu^{(t)}_j$ (see Lemma \ref{seq}).
%we set $\omega(\{i,j\})=(i,j).$ 
Then we set 
 $$\omega(\{i,j\}):=\begin{cases}(i,j) \text{ if } \mu^{(t)}_i< \mu^{(t)}_j \\ (j,i) \text{ if }\mu^{(t)}_i> \mu^{(t)}_j\end{cases}.$$
%If $\{i,j\}\in U'_{\mb}.$
Now we set $\overline{U}:=\omega(U_{\mb})\subset\zd.$ 
\par
Observe that for $1\le i\le r,$ $(r+1,i)\in\overline{U},$ hence $\pi_2(\overline{U})=[r].$

\begin{lem}\label{maggr}
Let $i,j\in [r],$ with $\mu_i\le\mu_j.$ Let $h>r$ be such that $(h,i)\in\omega({\mathcal T}).$ Then $(h,j)\in\omega({\mathcal T}).$
\end{lem}
\begin{proof}
Since $\{h,i\}\in{\mathcal T}$ and $h>r$ we get $h-r\in m(i)\cap m(r).$ On the other hand, $\mu_i\le\mu_j$ implies that $m(i)\cap m(r)\subseteq m(j)\cap m(r).$ 
Therefore $h-r\in m(j)\cap m(r),$ so $(h,j)\in\omega({\mathcal T}).$
\end{proof}

\begin{thm}\label{um}
$U_{\mb}$ is a generalized towerizable set.
\end{thm}
\begin{proof}
Let us consider $\overline{U}.$ It is enough to find a permutation $\tau$ on $[2r]$ such that $\tau(\overline{U})$ is a generalized tower set.
\par
We set $\overline{T}:=\omega({\mathcal T}).$ We want to show that the set $\{\overline{T}_i\}_{1\le i\le r}$ is totally ordered by inclusion. 
Let $i,j\in [r],$ $i\ne j.$
\par
Case $1:$ $\mu_i^{(k)}=\mu_j^{(k)}$ for every $k\ge 0.$ In this case we will show that $\overline{T}_i=\overline{T}_j.$ Of course it is enough to show that 
$\overline{T}_i\subseteq\overline{T}_j.$ Take $h\in\overline{T}_i$ i.e. $(h,i)\in\overline{T}.$ If $h>r$ by Lemma \ref{maggr} $(h,j)\in \overline{T}.$
If $h\le r$ let $t$ be the smallest integer such that $\mu_h^{(t)}<\mu_i^{(t)}=\mu_j^{(t)}.$ By the minimality of $t,$ $r_i^{(t)}=r_h^{(t)},$ hence
from $m(h)\cap m(r_h^{(t)})\subset m(i)\cap m(r_i^{(t)})=m(j)\cap m(r_j^{(t)})$ it follows that $j\not\in m(h).$ On the other hand if $h\in m(j),$ since   
$\mu_i^{(k)}=\mu_j^{(k)}$ for every $k\ge 0,$ by Lemma \ref{seq}, $\{i,j\}\not\in U'_{{\mathcal M}}$ i.e. either $i\in m(j)$ or $j\in m(i).$ 
If $j\in m(i)$ then $h\in m(i)$ and this contradicts that $\{h,i\}\in{\mathcal T}.$ If $i\in m(j),$ since both $h$ and $i$ belong to $m(j)$ then either 
$h\in m(i)$ or $i\in m(h),$ again a contradiction with $\{h,i\}\in{\mathcal T}.$ Therefore $h\not\in m(j).$ This together with $j\not\in m(h),$ as we saw, implies that $\{h,j\}\in{\mathcal T}.$ By the inequality $\mu_h^{(t)}<\mu_j^{(t)},$ we get $(h,j)\in\overline{T}.$
\par
Case $2:$ let $t$ be the smallest integer such that $\mu_i^{(t)}<\mu_j^{(t)}.$ We claim that $\overline{T}_i\subseteq\overline{T}_j.$ Take $h\in\overline{T}_i$ i.e. $(h,i)\in\overline{T}.$ If $h>r,$ using again Lemma \ref{maggr}, we are done. So we can assume $h\le r.$ Let $s$ be the smallest integer such that 
$\mu_h^{(s)}<\mu_i^{(s)}.$ Assume $s\le t$ (the same argument will work in the case $s>t$). We have to prove that $(h,j)\in\overline{T}.$ From the inequalities
$\mu_h^{(s)}<\mu_i^{(s)}\le\mu_j^{(s)}$ it is enough to show that $\{h,j\}\in{\mathcal T}$ i.e. $h\not\in m(j)$ and $j\not\in m(h).$ Using the same inequalities
we get 
 $$m(h)\cap m(r_h^{(s)})\subset m(i)\cap m(r_i^{(s)})\subseteq m(j)\cap m(r_j^{(s)}).$$
As above, by the minimality of $s,$ we have $r_h^{(s)}=r_i^{(s)}=r_j^{(s)},$ so we can deduce that $j\not\in m(h).$ On the other hand, if $h\in m(j),$ since there is $k\in m(j)\cap m(r_j^{(s)}),$ with $k\not\in m(h),$ we obtain that $h\in m(k)$ (note that we are using the fact that $h,k\in m(j)$ and $k\not\in m(h)$).
Since $m(k)\subseteq m(j)\cap m(r_j^{(s)}),$ we get $h\in m(r_j^{(s)}),$ so $h\in m(i),$ a contradiction, hence $h\not\in m(j)$ and we are done.
\par
Let $\tau$ be a permutation on $[2r],$ such that $\tau(i)<\tau(j)$ whenever $1\le i,j\le r$ and $\overline{T}_i\supseteq\overline{T}_j.$ Moreover $\tau(i)=i$ for $i\ge r+1.$ We denote by $S=\tau\omega(U_{\mb}).$ We want to show that $S$ is a generalized tower set i.e. we need to prove the four properties stated in Definition \ref{tg}.
\par
1) Clearly $S$ is connected by Lemma \ref{conn}. 
\par
2) Now set $T:=\tau(\overline{T})$ and $S_0=S\setminus T.$ By the properties of $\tau,$ $T$ is a tower set.
\par
3) Note that $\pi_2(S)=\pi_2(T)=[r]$ and $\pi_1(T)\subseteq [r]\cup\{r+1\le i\le 2r\mid i-r\in m(r)\}.$ Therefore if $(i,j)\in S_0,$ 
$\{\tau^{-1}(i),\tau^{-1}(j)\}\in U''_{\mb}.$ This implies that $i>r$ so $i\not\in\pi_2(T).$ On the other hand since 
$\{\tau^{-1}(i),\tau^{-1}(j)\}\not\in{\mathcal T},$ we have $i-r\not\in m(r)$ i.e. $i\not\in\pi_1(T).$ 
%Since $(i,\tau^{-1}(j))\not\in\overline{T}$ we get $j\in\pi_2(T).$
Let $u=\tau^{-1}(j).$ Then $(i,u)\not\in\overline{T},$ so $\{i,u\}\in U''_{\mb}.$ We claim that $(u,r)\in\overline{T},$ which will imply that $j\in\pi_1(T).$
We need to show that $u\not\in m(r).$ We know that $\{i,u\}\in U''_{\mb},$ so $i-r\in m(u).$ If $u\in m(r)$ then $i-r\in m(r)$ therefore $(i,j)\in T,$ a contradiction. Consequently $j\in\pi_1(T)\cap\pi_2(T).$
\par
4) Now we would like to prove that if $(i,\tau(j))\in S_0$ and $\tau(u)\in F_T(\tau(j))$ then $(i,\tau(u))\in S_0,$ i.e.  
$\{i,u\}\in{\mathcal S}_0$ hence we have to show that $i-r\in m(u)$ and $i-r\not\in m(r).$ Since $(i,\tau(j))\in S_0$ we have $i-r\in m(j)\setminus m(r).$ 
%By definition we have to prove that $i-r\in m(u).$ 
%Since $\{i, j\}\in {\mathcal S}_0$ and $j \le r< i,$  we get $$i-r\in  m(j).$$
Since  $\tau(u) \in F_T(\tau(j))$ we have $ T_{\tau(j)}\subset T_{\tau(u)}$ and $ (\tau(j),\tau(u))\not\in T.$ From this we deduce that 
$\{u,j\}\not\in U'_{\mb},$ consequently either $j\in m(u)$ or $u\in m(j).$ If $j\in m(u)$ since $i-r\in m(j)$ we get $i-r\in m(u)$ and we are done.
If $u\in m(j)$ since also $i-r\in m(j)$ we get either $u\in m(i-r)$ or $i-r\in m(u).$ 
If $u\in m(i-r)$ and $i-r\not\in m(u),$ take $v\in\overline{T}_u\setminus\overline{T}_j$ and consider $\{u,v\},$ $\{i,j\}\in U_{\mb}.$ Since $U_{\mb}$ is connected we have that one of the following sets must belong to $U_{\mb}:$
$$\{u,i\},\,\{u,j\},\,\{v,i\},\,\{v,j\}.$$
Note that $\{u,i\}\not\in U_{\mb}$ since $i-r\not\in m(u).$ Moreover $\{u,j\}\not\in U_{\mb}$ (see above). 
If $\{v,i\}\in U_{\mb},$ then $(i,v)\in\omega({\mathcal S}_0),$ so $i-r\in m(v),$ consequently $u\in m(v),$ which contradicts that $\{u,v\}\in U_{\mb}.$
\par
If $\{v,j\}\in U_{\mb},$ then $(j,v)\in\omega({\mathcal S}_0).$ Therefore $j\in\overline{T}_v\setminus\overline{T}_u$ and 
$v\in\overline{T}_u\setminus\overline{T}_v,$ a contradiction since $\overline{T}_u$ and $\overline{T}_v$ are comparable by inclusion. 
Consequently $i-r\in m(u)$ and we are done.
\end{proof}
\begin{rem} We note that  if a Hilbert-Burch matrix of standard form $\mb=(m_{i,j})$  is  bidiagonal (i.e. the only non zero entries are $m_{i,i}$ and $m_{i, i+1}$) then $U_{\mb}$ is a towerizable set. This follows by Theorem \ref{um}, since, using the same notation as above, $\sd_0=\emptyset.$ Vice versa if $U_{\mb}$ is a towerizable set then it is easy to build a bidiagonal Hilbert-Burch matrix of standard form $\mb'$ such that $I(\mb)=I(\mb').$ This generalizes a result of Ahn and Shin in \cite{AS}.
\end{rem}
In order to get our main result we need the following Lemma.

\begin{lem}\label{dept}
With the same terminology as above, let $\mb$ be a Hilbert-Burch matrix of standard form. Let $\{i, j\}, \{u,v\}$ be two different elements in $U_{\mb}.$ Then, if we set $H_i:=D_i$ and $H_{r+i}:=M_i$ for $i=1,..,r,$ 
$$\depth (H_i, H_j, H_u, H_v)\geq 3.$$
\end{lem}
\begin{proof} We need to distinguish three possibilities. If $\{i, j\}, \{u,v\}\in U'_{\mb}$ then $i, j, u, v\le r,$  since  $\prod_{h=1..r} D_h$ is a minimal generator of $I(\mb)$ and $|\{i, j, u,v \}|\ge 3$ we are done. 
If $\{i, j\}, \{u,v\}\in U''_{\mb},$ say $i, u\ge r,$ then if $j=v$ we have $i-r\in m(j)$ and $u-r\in m(j).$ So, by Proposition \ref{HBgen}, $H_iH_u$ is a factor of $f_j$ hence $H_i, H_u$ are coprime as $f_j$ is squarefree, then  $\depth (H_i, H_j, H_u)=3.$     
If $j<v,$ then  $i-r\in m(j)$ and $v\notin m(j).$ So, by Proposition \ref{HBgen}, $H_iH_v$ is a factor of $f_j$ hence $H_i, H_v$ are coprime as $f_j$ is squarefree, then  $\depth (H_i, H_j, H_v)=3.$ 
If $\{i, j\}\in U'_{\mb}, \{u,v\}\in U''_{\mb},$ with $v<u,$ when $|\{i, j, v \}|= 3$ we are done, otherwise say $v=j$ then $H_uH_i$ is a factor of $f_j$ so  $\depth (H_i, H_j, H_u)=3.$
\end{proof}

Collecting all the results of this section we are ready to proof the main result.

\begin{thm}\label{gts-acm}
Let $I\subset k[x_1,\ldots,x_n]$ be a monomial squarefree of height $2.$ Then $I$ is aCM iff it defines a generalized tower scheme. 
\end{thm}
\begin{proof}
Let us suppose that $I$ is aCM. Then $I=I(\mb)$ for some $\mb$ of standard form of size $(r+1)\times r$ (see Lemma \ref{mat} and Definition \ref{dfst}). By Theorem \ref{um} $U_{\mb}$ is a generalized towerizable set. Let $\omega$ and $\tau$ be as in the proof of Theorem \ref{um}. Let $S:=\tau(\omega(U_{\mb})),$ which is a generalized tower set. For $j\in\pi_1(S),$ we set
 $$f_{1j}:=\begin{cases}D_{\tau^{-1}(j)}& \text{ for }j\le r \\ M_{j-r} & \text{ for }j> r\end{cases}.$$
For $j\in\pi_2(S)=[r],$ we set $f_{2j}:=D_{\tau^{-1}(j)}.$
Also we set $$\ff_i:=\{f_{ij}\mid j\in\pi_i(S)\},\, i=1,2.$$
By Lemma \ref{dept}, $\ff_1$ and $\ff_2$ satisfy the conditions of genericity required by Definition \ref{gts}. 
%In fact, since for every $(i,j)\in S,$ 
%$(f_{1i},f_{2j})\supseteq I,$ $f_{1i},f_{2j}$ are coprime. Moreover if $(i,j),$ $(h,k)\in S,$ with $\{i,j\}\ne\{h,k\}$
\par
We claim that $I_S(\ff_1,\ff_2)=I.$ At first we show that $I\subseteq I_S(\ff_1,\ff_2).$ Indeed, let $g_k$ be the maximal minor obtained from $\mb$ by deleting the $k$-th row and take any $(i,j)\in S.$ We need to show that $g_k\in (f_{1i},f_{2j})$ for every $k.$ Since 
$g_0=\prod_{1\le i\le r} D_i$ and $f_{2j}=D_{\tau^{-1}(j)},$ $g_0\in (f_{1i},f_{2j}).$ Assume $k\ge 1.$ Since $(i,j)\in S$ we have $\{\tau^{-1}(i),\tau^{-1}(j)\}\in U_{\mb}=U'_{\mb}\cup U''_{\mb}.$ If $\{\tau^{-1}(i),\tau^{-1}(j)\}\in U'_{\mb}$ then $(f_{1i},f_{2j})=(D_{\tau^{-1}(i)},D_{\tau^{-1}(j)})$ and by Proposition \ref{mdcfr} we get either $\tau^{-1}(i)\not\in m(k)$ or $\tau^{-1}(j)\not\in m(k)$ therefore by Proposition \ref{HBgen}, $g_k\in (f_{1i},f_{2j}).$ We proceed analogously if $\{\tau^{-1}(i),\tau^{-1}(j)\}\in U''_{\mb}.$
\par
Now we show that $I_S(\ff_1,\ff_2)\subseteq I.$ Let $g\in I_S(\ff_1,\ff_2)$ be a squarefree monomial. Let
 $$E_g=\{0\}\cup\{1\le h\le r \mid g\in (M_h) \text{ and }g\not\in (D_h)\}.$$
We set $e:=\max E_g.$ We claim that $g\in (g_e).$
\par
Assume $e=0$ and let $h\in [r].$ If $g\in (M_h)$ then $g\in (D_h)$ by the maximality of $e.$ 
If $g\not\in (M_h)$ then $g\in (D_h)$ since $\{r+h,h\}\in U''_{\mb.}$
\par
Assume $e>0.$ Remind that $g_e=\prod_{j\in m(e)} M_j \cdot \prod_{j\not\in m(e)} D_j.$ If $g\in (M_h)$ then $g\in (D_h)$ for $h>e$ by the maximality of $e.$ 
If $g\not\in (M_h)$ then $g\in (D_h)$ since $\{r+h,h\}\in U''_{\mb.}$
Note that if $e=1$ we are done. So we can assume $e>1.$ If $h\not\in m(e)$ and $h<e$ then $\{e,h\}\in U'_{\mb},$ consequently $g\in (D_e,D_h),$ but 
$g\not\in (D_e)$ so $g\in (D_h).$ Now let $h\in m(e),$ 
namely $\{r+h,e\}\in U''_{\mb.}$ Therefore $g\in (M_h,D_e),$ but $g\not\in (D_e)$ so $g\in (M_h).$ Since $g_e$ is a squarefree monomial, 
we showed that $g\in (g_e).$
\par
Vice versa let us suppose that $I$ defines a generalized tower scheme. This means that $I=I_S(\ff_1,\ff_2)$ where $S$ is a generalized tower set and $\ff_1,\ff_2$ are families of monomials, satisfying the conditions of Definition \ref{gts}. By Theorem \ref{gtsacm}, $S$ is aCM. So we get that $I$ is aCM just applying Proposition \ref{acmci}.

\end{proof}

\vspace{1cm}
%\newpage
%\markboth{}{}
{\f
{\sc (G. Favacchio) Dip. di Matematica e Informatica, Universit\`a di Catania,\\
                  Viale A. Doria 6, 95125 Catania, Italy}\par
{\it E-mail address: }{\tt favacchio@dmi.unict.it} \par
{\it Fax number: }{\f +39095330094} \par
\vspace{.3cm}
{\f
{\sc (A. Ragusa) Dip. di Matematica e Informatica, Universit\`a di Catania,\\
                  Viale A. Doria 6, 95125 Catania, Italy}\par
{\it E-mail address: }{\tt ragusa@dmi.unict.it} \par
{\it Fax number: }{\f +39095330094} \par
\vspace{.3cm}
{\sc (G. Zappal\`a) Dip. di Matematica e Informatica, Universit\`a di Catania,\\
                  Viale A. Doria 6, 95125 Catania, Italy}\par
{\it E-mail address: }{\tt zappalag@dmi.unict.it} \par
{\it Fax number: }{\f +39095330094}
}


\begin{thebibliography}{HMNW}
\markboth{\it References}{\it References}



%\bibitem[A]{A}
%  M.~Amasaki,
%  \emph{The weak Lefschetz property for Artinian graded rings and basic sequences},
%  arXiv:1109.2365v1 [math.AC].


%\bibitem[BE]{BE}
%  D.~A.~Buchsbaum, D.~Eisenbud,
%  \emph{Algebra structures for finite free resolutions, and some structure
%  theorems for ideals of codimension 3},
%  Amer. J. Math. \textbf{99}(1) (1977), 447--485.

%\bibitem[Bi]{Bi} 	
%A.~M.~Bigatti,
%\emph{Upper Bounds For The Betti Numbers Of A Given Hilbert Function}, 
%Comm. Algebra 21 (1993), no. 7, 2317--2334.

  
%\bibitem[BK]{BK}
 % H.~Brenner, A.~Kaid,
 % \emph{Syzygy bundles on $\pp^2$ and the Weak Lefschetz property},
 % Ill. J. Math. \textbf{51}(4) (2007), 1299--1308.
\bibitem[AS]{AS}  
J. Ahn, and Y. S. Shin. \emph{The minimal free resolution of a star-configuration in $\pp^n$ and the Weak Lefschetz Property.} J. Korean Math. Soc 49.2 (2012): 405-417.  
  
\bibitem[GHM]{GHM}
 A. V. Geramita, B. Harbourne, and J. Migliore. \emph{Star configurations in $\pp^n$. }
 J. of Algebra \textbf{376} (2013): 279-299.

\bibitem[GHS]{GHS}
 A. V.  Geramita, T.Y. Harima, S. Shin, \emph{Extremal Point Sets and Gorenstein Ideals}. 
 Adv. in Math. \textbf{152}(1) (2000), 78-119.

\bibitem[GS]{GS}
 A. V.  Geramita, S. Shin, \emph{{$k$}-configurations in {$\pp^3$} all have extremal resolutions}. 
 J. of Algebra \textbf{213} (1999): 351--368.


  
%\bibitem[Bu]{Bu}
 % L.~Burch,
  %\emph{On ideals of finite homological dimension in local rings},
   %Proc. Cambridge Philos. Soc. \textbf{64} (1968), 941--948.
   
%\bibitem[Ca]{Ca}
 %G.~Campanella,
  %\emph{Standard bases of perfect homogeneous polynomial ideals of height 2},
  %J. of Alg. \textbf{101}(1) (1986), 47--60.
 
 %\bibitem[CV]{CV}
 %A.~Conca, G.~Valla,
  %\emph{Betti numbers and lifting of Gorenstein codimension three ideals},
  %Comm. Algebra \textbf{28} (2000), no. 3, 1371--1386.
  
%\bibitem[Da]{Da}
% E.~D.~Davis,
 %\emph{Hilbert functions and complete intersections},
 %Rend. Sem. Mat. Univ. Politec. Torino 42 (1984), no. 2, 25--28 
 
% \bibitem[CN]{CN}
%D.~Cook, U.~Nagel,
%  \emph{The weak Lefschetz property, monomial ideals, and lozenges},
%  To appear in Illinois J. Math.
 
% \bibitem[Di]{Di}
%S.~Diesel,
%  \emph{Irreducibility and dimension theorems for families of height 3
%  Go\-ren\-stein algebras},
%  Pac. J. of Math. \textbf{172}(4) (1996), 365--397.

 
%\bibitem[E]{E}
% D.~Eisenbud, 
% \emph{Commutative Algebra with a View Toward Algebraic Geometry},
% Springer-Verlag, GTM 150 (1996).
 
%\bibitem[Ga]{Ga}
% F.~Gaeta,
% \emph{Quelques progr\ac es r\'ecents dans la classification des vari\'et\'es alg\'ebriques d'un espace projectif},
% Deuxi\ac eme Colloque de G\'eom\'etrie Alg\'ebrique, Li\ac ege, (1952), pp. 145--183.
 
%\bibitem[GM]{GM}
% A.~Geramita, J.~Migliore,
 % \emph{Reduced Gorenstein codimension three subschemes of projective space},
  %Proc. Amer. Math. Soc. \textbf{125}(4) (1997), 943--950.
 
%\bibitem[Ha]{Ha}
% R.~Hartshorne, 
% \emph{Algebraic geometry},
% Springer-Verlag, GTM 52 (1977).
 
%\bibitem[HMNW]{HMNW}
% T.~Harima, J.C.~Migliore, U.~Nagel, J.~Watanabe,
 % \emph{The Weak and Strong Lefschetz properties for Artinian K-algebras},
 % J. Alg. \textbf{262} (2003), 99--126.

%\bibitem[HP]{HP}
%J. ~Herzog and D. ~Popescu,
%\emph{The strong Lefschetz property and simple extensions}, preprint. Available
%on the arXiv at http://front.math.ucdavis.edu/0506.5537,
 

%\bibitem[Hu]{Hu}	
%H.~Hulett, 
%\emph{Maximum Betti numbers for a given Hilbert function}, 
%Comm. Algebra 21 (1993), no. 7, 2335--2350.

 
%\bibitem[Ik]{Ik}
%H. ~Ikeda, 
%\emph{Results on Dilworth and Rees numbes of Artinian local rings}, Japan. J. Math. \textbf{22} (1996), 147--158.

%\bibitem[Mi]{Mi}
%J.~Migliore, 
% \emph{The geometry of the weak Lefschetz property and level sets of points},
%Canad. J. Math. (2008), n.2, 391--411

%\bibitem[MZ]{MZ}
%J.~Migliore, F. Zanello,
%  \emph{The Hilbert functions which force the weak Lefschetz property},
%J. Pure Appl. Algebra 210 (2007), n.2, 465--471

%\bibitem[MMN]{MMN}
%J.~Migliore, R.~Mir�-Roig, U.~Nagel,
%  \emph{Monomial ideals, almost complete intersections and the weak Lefschetz property},
%Trans. Amer. Math. Soc.  363  (2011),  no. 1, 229�257.

%\bibitem[MN]{MN}
%J.~Migliore, U.~Nagel,
%  \emph{A tour of the Weak and Strong Lefschetz Properties},
%arXiv:1109.5718 

%\bibitem[MZ]{MZ}
%J.~Migliore, F.~Zanello,
%  \emph{The strength of the weak Lefschetz property},
%Illinois J. Math. 52 (2008), no. 4, 1417--1433.
  
\bibitem[MR]{MR}
 R.~Maggioni, A.~Ragusa,
  \emph{Connections between Hilbert function and geometric properties for a finite set of points in ${\pp}\sp 2$}, 
 Le Matematiche, \textbf{39} (1984), no. 1-2, 153--170.

%\bibitem[Pa]{Pa}	
%K.~Pardue, 
%\emph{ Deformation Classes Of Graded Modules And Maximal Betti Numbers}, 
%Ill. J.Math. \textbf{40} (1996), 564--585.
  
%\bibitem[PS]{PS}
%C.~Peskine, L.~Szpiro,
%  \emph{Liaison des vari\'et\'es alg\'ebriques. I},
%  Inv. Math. \textbf{26} (1974), 271--302.
   
  
% \bibitem[RZ2]{RZ2}
% A. Ragusa, G. Zappal\ac a, 
 %\emph{Gorenstein schemes on general surfaces},
 %   Nagoya  Math. J. Vol.162, 111-125 (2001)
 
%\bibitem[RZ4]{RZ4}
%A. Ragusa, G. Zappal\ac a, 
%\emph{On the reducibility of the postulation Hilbert scheme},
%Rend. Circ. Mat. Palermo (2)  53  (2004),  no. 3, 401�--406.
    
%\bibitem[RZ2]{RZ2}
%A. Ragusa, G. Zappal\ac a, 
%\emph{On the Weak-Lefschetz property for Artinian Gorenstein algebras},
%arXiv:1112.1498. To appear in Rend. Circ. Mat. Palermo.

%\bibitem[RZ3]{RZ3}
%A. Ragusa, G. Zappal\ac a, 
%\emph{Properties of $3$-codimensional Gorenstein schemes},
%Comm. Algebra 29 (2001), no. 1, 303--318.
   
 %\bibitem[RZ3]{RZ3}
%A. Ragusa, G. Zappal\ac a, 
%\emph{Subschemes of special determinantal and pfaffian projective schemes}
 %  Preprint. 
 
%\bibitem[RZ2]{RZ2}
%A.~Ragusa, G.~Zappal\ac a, 
%\emph{On Complete Intersections Contained in Cohen-Macaulay
%and Gorenstein ideals}
%To appear in Algebra Colloquium. 

%\bibitem[RRR]{RRR}
%L. ~Reid, L. ~Roberts and M. ~Roitman,
%  \emph{On complete intersections and their Hilbert functions},
% Canad. Math. Bull. \textbf{34} (4) (1991), 525--535.

%\bibitem[St1]{St1}
%R.~Stanley,
%  \emph{Hilbert functions of graded algebras},
%  Adv. in Math. \textbf{28} (1978), 57--83.
 
%\bibitem[St]{St}
%R. ~Stanley, \emph{Weyl groups, the hard Lefschetz theorem, and the Sperner property}, SIAM J. Algebraic
%Discrete Methods \textbf{1} (1980), 168--184.

%\bibitem[W]{W}
%A.~Wiebe, \emph{The Lefschetz property for componentwise linear ideals and Gotzmann ideals}, 
%Comm. Algebra  32  (2004),  no. 12, 4601�--4611.


%\bibitem[W]{W}
%J.~Watanabe,
%  \emph{The Dilworth number of Artinian rings and finite posets with rank function},
%  Commutative Algebra and Combinatorics, Advanced Studies in Pure Math. \textbf{11} (1987), 303--312.

\bibitem[RZ]{RZ1}
A.~Ragusa, G.~Zappal\`a,
  \emph{Partial intersection and graded Betti numbers},
Beitr\"age zur Algebra und Geometrie, \textbf{44} (2003), no. 1, 285--302.








\end{thebibliography}
\end{document}